\documentclass[12pt]{amsart}
\usepackage{enumerate,amssymb,version,aliascnt,geometry,stmaryrd,mathrsfs,chngcntr}
\usepackage[all]{xy}
\usepackage[mathscr]{euscript}
\usepackage{hyperref}

\hypersetup{
pdfauthor={Olivier Haution},
pdftitle={Odd rank vector bundles in eta-periodic motivic homotopy theory},
pdfkeywords={},
hidelinks,
unicode
}

\usepackage[T1]{fontenc}
\usepackage{textcomp}

\let\origsetminus\setminus
\let\originfty\infty
\let\origpartial\partial
\let\origin\in
\let\origsubset\subset
\let\origsimeq\simeq
\let\origast\ast
\let\origsum\sum
\usepackage{mathabx}
\let\setminus\origsetminus
\let\infty\originfty
\let\partial\origpartial
\let\in\origin
\let\subset\origsubset
\let\simeq\origsimeq
\let\ast\origast
\let\sum\origsum

\swapnumbers

\DeclareMathOperator{\Th}{Th}
\DeclareMathOperator{\thom}{th}
\DeclareMathOperator{\SH}{SH}
\DeclareMathOperator{\W}{W}
\DeclareMathOperator{\Hop}{H_\bullet}
\DeclareMathOperator{\Hot}{H}
\DeclareMathOperator{\Spc}{Spc}
\DeclareMathOperator{\Spcp}{Spc_\bullet}
\DeclareMathOperator{\Spt}{Spt}

\DeclareMathOperator{\SL}{SL}
\DeclareMathOperator{\Ho}{H}
\DeclareMathOperator{\Aut}{Aut}
\DeclareMathOperator{\End}{End}
\DeclareMathOperator{\Pic}{Pic}
\DeclareMathOperator{\can}{can}
\DeclareMathOperator{\id}{id}
\DeclareMathOperator{\Spec}{Spec}
\DeclareMathOperator{\Hom}{Hom}
\DeclareMathOperator{\BGL}{BGL}
\DeclareMathOperator{\BSL}{BSL}
\DeclareMathOperator{\GL}{GL}
\DeclareMathOperator{\Oo}{O}
\DeclareMathOperator{\Gr}{Gr}
\DeclareMathOperator{\hocolim}{hocolim}
\DeclareMathOperator{\rank}{rank}
\DeclareMathOperator{\Sm}{Sm}
\DeclareMathOperator{\Ab}{\mathbb{A}}
\DeclareMathOperator{\colim}{colim}

\newcommand{\T}{T}
\newcommand{\vo}{\oplus}
\newcommand{\Uni}{\mathcal{U}}
\newcommand{\Quo}{\mathcal{Q}}
\newcommand{\BB}{\mathrm{B}}
\newcommand{\Un}{\mathbf{1}}
\newcommand{\Aa}{A^{*,*}}
\newcommand{\Su}{\Sigma^{\infty}}
\newcommand{\Sup}{\Sigma^{\infty}_+}
\newcommand{\Oc}{\mathcal{O}}
\newcommand{\mun}{\mu_n}

\newcommand{\Zz}{\mathbb{Z}}
\newcommand{\Pp}{\mathbb{P}}
\newcommand{\Nn}{\mathbb{N}}
\newcommand{\Gm}{\mathbb{G}_m}

\newcommand{\Dc}{\mathcal{D}}

\newlength{\mylen}
\setbox1=\hbox{$\bullet$}\setbox2=\hbox{\tiny$\bullet$}
\newtheorem*{theorem*}{Theorem}
\newtheorem*{proposition*}{Proposition}

\newtheorem{theorem}{Theorem}[section]
\newaliascnt{proposition}{theorem}
\newtheorem{proposition}[proposition]{Proposition}
\newaliascnt{lemma}{theorem}
\newtheorem{lemma}[lemma]{Lemma}
\newaliascnt{corollary}{theorem}
\newtheorem{corollary}[corollary]{Corollary}

\theoremstyle{definition}
\newaliascnt{remark}{theorem}
\newtheorem{remark}[theorem]{Remark}
\newaliascnt{example}{theorem}
\newtheorem{example}[example]{Example}
\newaliascnt{definition}{theorem}
\newtheorem{definition}[definition]{Definition}

\newtheoremstyle{par}
  {}
  {}
  {}
  {}
  {}
  {.}
  { }
  {}%
\theoremstyle{par}
\newtheorem{para}[theorem]{}

\counterwithin{equation}{theorem}
\numberwithin{equation}{theorem}

\newcommand{\rref}[1]{(\ref{#1})}
\newcommand{\dref}[2]{(\ref{#1}.\ref{#2})}

\begin{document}
\begin{abstract}
We observe that, in the eta-periodic motivic stable homotopy category, odd rank vector bundles behave to some extent as if they had a nowhere vanishing section. We discuss some consequences concerning $\SL^c$-orientations of motivic ring spectra, and the \'etale classifying spaces of certain algebraic groups. In particular, we compute the classifying spaces of diagonalisable groups in the eta-periodic motivic stable homotopy category.
\end{abstract}

\author{Olivier Haution}
\title{Odd rank vector bundles in eta-periodic motivic homotopy theory}
\email{olivier.haution at gmail.com}
\address{Mathematisches Institut, Ludwig-Maximilians-Universit\"at M\"unchen, Theresienstr.\ 39, D-80333 M\"unchen, Germany}
\thanks{This work was supported by the DFG research grant HA 7702/5-1 and Heisenberg grant HA 7702/4-1.}


\date{\today}

\maketitle

\numberwithin{theorem}{section}
\numberwithin{lemma}{section}
\numberwithin{proposition}{section}
\numberwithin{corollary}{section}
\numberwithin{example}{section}
\numberwithin{definition}{section}
\numberwithin{remark}{section}

\section*{Introduction}
Around forty years ago, Arason computed the Witt groups of projective spaces \cite{Arason-Witt-Pn}. This computation was later revisited by Gille \cite{Gille-Witt-Pn}, Walter \cite{Walter-GW-Pn}, and Nenashev \cite{Nenashev-Witt-Pn}. It exhibited Witt groups as a somewhat exotic cohomology theory, whose value on projective spaces differs quite drastically from what is obtained in more classical cohomology theories such as Chow groups or $K$-theory. It is now understood that this behaviour reflects the lack of $\GL$-orientation in Witt theory.

Ananyevskiy observed in  \cite{Ana-Pushforwards} that the key property of Witt groups permitting to perform these computations turns out to be the fact that the Hopf map
\[
\eta \colon \Ab^2 \smallsetminus \{0\} \to \Pp^1, \quad (x:y) \mapsto [x:y]
\]
induces by pullback an isomorphism of Witt groups $\W(\Pp^1) \xrightarrow{\sim} \W(\Ab^2 \smallsetminus \{0\})$. He thus extended in \cite{Ana-Pushforwards} the above-mentioned computations to arbitrary cohomology theories in which $\eta$ induces an isomorphism.\\

Inverting the Hopf map $\eta$ in the motivic stable homotopy category $\SH(S)$ over a base scheme $S$ yields its $\eta$-periodic version $\SH(S)[\eta^{-1}]$, a category which has been studied in details by Bachmann--Hopkins \cite{Bachmann-Hopkins}. In this paper, we lift Ananyevskiy's computations of the cohomology of projective spaces to the $\eta$-periodic stable homotopy category: we obtain for instance that a projective bundle of even relative dimension becomes an isomorphism in $\SH(S)[\eta^{-1}]$.

Another familiar feature of Witt groups is that twisting these groups by squares of line bundles has no effect, which may be viewed as a manifestation of the $\SL^c$-orientability of Witt groups (see below). We show that that property of Witt groups in fact follows from their $\eta$-periodicity alone (see \rref{prop:twist_square} for a more general statement):
\begin{proposition*}
Let $V \to X$ be a vector bundle and $L \to X$ a line bundle. Then we have an isomorphism of Thom spaces
\[
\Th_X(V) \simeq \Th_X(V \otimes L^{\otimes 2}) \in \SH(S)[\eta^{-1}].
\]
\end{proposition*}

Panin and Walter introduced \cite[\S3]{PW-BO} the notion of $\SL^c$-orientability for algebraic cohomology theories, which consists in the data of Thom classes for vector bundles equipped with a square root of their determinant, and proved that Hermitian $K$-theory is $\SL^c$-oriented. Ananyevskiy later showed \cite[Theorem~1.2]{Ana-SL} that a cohomology theory is $\SL^c$-oriented as soon as it is a Zariski sheaf in bidegree $(0,0)$, and pointed out \cite[Theorem~1.1]{Ana-SL} the close relations between $\SL^c$-orientations and $\SL$-orientations (the latter consisting in the data of Thom classes for vector bundles with trivialised determinant). We show in this paper that the two notions actually coincide in the $\eta$-periodic context:

\begin{theorem*}
Every $\SL$-orientation of an $\eta$-periodic motivic commutative ring spectrum is induced by a unique $\SL^c$-orientation.
\end{theorem*}

These results are obtained as consequences of the following observation:
\begin{proposition*}
Let $E$ be a vector bundle of odd rank over a smooth $S$-scheme $X$, and $E^\circ$ the complement of the zero-section in $E$. 
\begin{enumerate}[(i)]
\item The projection $E^\circ \to X$ admits a section in $\SH(S)[\eta^{-1}]$.

\item The diagram $E^\circ \times_X E^\circ \rightrightarrows E^\circ \to X$ becomes a split coequaliser diagram in $\SH(S)[\eta^{-1}]$.
\end{enumerate}
\end{proposition*}
The first assertion may be viewed as a splitting principle, while the second permits to perform a form of descent. To some extent, this proposition allows us to assume that odd rank vector bundle admit a nowhere-vanishing section (once $\eta$ is inverted); in particular that line bundles are trivial.\\

Finally, we provide applications to the computation in $\SH(S)[\eta^{-1}]$ of the \'etale classifying spaces of certain algebraic groups.
\begin{theorem*}
For $r \in \Nn \smallsetminus \{0\}$, there exist natural maps
\[
S \to \BB\Gm \quad ;\quad S \to \BB \mu_{2r+1} \quad ; \quad \Gm \to \BB \mu_{2r}
\]
which become isomorphisms in $\SH(S)[\eta^{-1}]$.
\end{theorem*}
From this theorem, we deduce a computation in $\SH(S)[\eta^{-1}]$ of the classifying space of an arbitrary diagonalisable group. We obtain that all invariants of torsors under a diagonalisable group with values in an $\eta$-periodic cohomology theory arise from a single invariant of $\mu_2$-torsors. In the appendix, we present an explicit construction of that invariant, exploiting the identification of the group $\mu_2$ with the orthogonal group $\Oo_1$.\\

Next, we obtain ``relative'' computations in $\SH(S)[\eta^{-1}]$ of certain \'etale classifying spaces in terms of others:
\begin{theorem*}
For $n\in \Nn\smallsetminus \{0\}$ and $r\in \Nn$, the natural morphisms
\[
\BSL_n \to \BSL^c_n \quad ; \quad \BGL_{2r} \to \BGL_{2r+1} \quad ; \quad \BSL_{2r+1} \to \BGL_{2r+1}
\]
become isomorphisms in $\SH(S)[\eta^{-1}]$.
\end{theorem*}
The first result can be viewed as a companion of the theorem on orientations stated above, and cements the idea that the groups $\SL_n^c$ and $\SL_n$ are the same in the eyes of $\eta$-periodic stable homotopy theory. The second (resp.\ third) result expresses the fact that odd-dimensional vector bundles behave as if they had a nowhere-vanishing section (resp.\ trivial determinant) from the point of view of $\eta$-periodic stable homotopy theory.

The morphism $\BSL_{2r} \to \BGL_{2r}$ is not an isomorphism in $\SH(S)[\eta^{-1}]$, but we show that it admits a section, expressing the fact that every invariant (with values in an $\eta$-periodic cohomology theory) of even-dimensional vector bundles is determined by its value on those bundles having trivial determinant.\\

Finally, let us mention that the results of this paper serve as a starting point for the paper \cite{hyp} on Pontryagin classes.\\

\noindent \textbf{Acknowledgments}.
I am grateful to the referee for his/her suggestions, and in particular for noticing that the coequaliser diagram of \rref{cor:split_line} is split.

\section{Notation and basic facts}
\begin{para}
Throughout the paper, we work over a noetherian base scheme $S$ of finite dimension. The category of smooth separated $S$-schemes of finite type will be denoted by $\Sm_S$. All schemes will be implicitly assumed to belong to $\Sm_S$, and the notation $\Ab^n,\Pp^n,\Gm$ will refer to the corresponding $S$-schemes. We will denote by $1$ the trivial line bundle over a given scheme in $\Sm_S$.
\end{para}

\begin{para}
We will use the $\Ab^1$-homotopy theory introduced by Morel--Voevodsky \cite{MV-A1}. We will denote by $\Spc(S)$ the category of motivic spaces (i.e.\ simplicial presheaves on $\Sm_S$), by $\Spcp(S)$ its pointed version, and by $\Spt(S)$ the category of $\T$-spectra, where $\T= \Ab^1/\Gm$. We endow these with the motivic equivalences, resp.\ stable motivic equivalences, and denote by $\Hot(S), \Hop(S), \SH(S)$ the respective homotopy categories. We refer to e.g.\ \cite[Appendix~A]{PMR} for more details.

We have an infinite suspension functor $\Su\colon \Spcp(S) \to \Spt(S)$. Composing with the functor $\Spc(S) \to \Spcp(S)$ adding an external base-point, we obtain a functor $\Sup \colon \Spc(S) \to \Spt(S)$.

The spheres are denoted as usual by $S^{p,q} \in \Spcp(S)$ for $p,q\in \Nn$ with $p\geq q$ (where $\T\simeq S^{2,1}$). The motivic sphere spectrum $\Sup S$ will be denoted by $\Un_S \in \Spt(S)$. When $A$ is a motivic spectrum, we denote its $(p,q)$-th suspension by $\Sigma^{p,q}A = S^{p,q} \wedge A$. This yields functors $\Sigma^{p,q} \colon \SH(S) \to \SH(S)$ for $p,q\in \Zz$.
\end{para}

\begin{para}
When $E\to X$ is a vector bundle with $X \in \Sm_S$, we denote by $E^{\circ}=E \smallsetminus X$ the complement of the zero-section. The Thom space of $E$ is the pointed motivic space $\Th_X(E) = E/E^\circ$. We will write $\Th_X(E) \in \Spt(S)$ instead of $\Su \Th_X(E)$, in order to lighten the notation. When $g \colon Y \to X$ is a morphism in $\Sm_S$, we will usually write $\Th_Y(E)$ instead of $\Th_Y(g^*E)$. Since $E \to X$ is a weak equivalence, we have a cofiber sequence in $\Spcp(S)$, where $p \colon E^\circ \to X$ is the projection,
\begin{equation}
\label{eq:E_circ_dist}
(E^\circ)_+ \xrightarrow{p_+} X_+ \to \Th_X(E).
\end{equation}
If $F \to S$ is a vector bundle and $f \colon X \to S$ the structural morphism, we have by \cite[Proposition~3.2.17~(1)]{MV-A1} a natural identification in $\Spcp(S)$
\begin{equation}
\label{eq:Thom_sum}
\Th_X(E \oplus f^*F) = \Th_X(E) \wedge \Th_S(F).
\end{equation}

When $V \to S$ is a vector bundle, we denote by $\Sigma^V \colon \SH(S) \to \SH(S)$ the derived functor induced by $A \mapsto A \wedge \Th_S(V)$. It is an equivalence of categories, with inverse denoted by $\Sigma^{-V}$.
\end{para}

\begin{para}
\label{p:purity}
Let $i \colon Y\to X$ be a closed immersion in $\Sm_S$, with normal bundle $N \to Y$, and open complement $u\colon U \to X$. The purity equivalence $X/U \simeq \Th_Y(N)$ (see e.g.\ \cite[Theorem~3.2.23]{MV-A1}) yields a cofiber sequence in $\Spcp(S)$
\begin{equation}
\label{eq:purity}
U_+ \xrightarrow{u_+} X_+ \to \Th_Y(N).
\end{equation}
More generally, if $V \to X$ is a vector bundle, we have a  cofiber sequence in $\Spcp(S)$
\[
\Th_U(V) \to \Th_X(V) \to \Th_Y(N \oplus i^*V).
\]
This may be deduced from \eqref{eq:purity} by first reducing to the case $X=S$ using the functor $f_\sharp$ of \rref{p:sharp} below, and then applying the functor $- \wedge \Th_X(V)$ (both of which preserve homotopy colimits), in view of \eqref{eq:Thom_sum}.
\end{para}

\begin{para}
\label{p:Thom_funct}
Let $\varphi \colon E \xrightarrow{\sim} F$ be an isomorphism of vector bundles over $X \in \Sm_S$. Then $\varphi$ induces a weak equivalence in $\Spcp(S)$ (and $\Spt(S)$)
\[
\Th(\varphi) \colon \Th_X(E) \to \Th_X(F).
\]
If $\psi \colon F \xrightarrow{\sim} G$ is an isomorphism of vector bundles over $X$, we have
\begin{equation}
\label{eq:T_mult}
\Th(\psi \circ \varphi) = \Th(\psi) \circ \Th(\varphi).
\end{equation}

If $X=S$ and $f\in \End_{\SH(S)}(\Un_S)$, then we have in $\SH(S)$
\begin{equation}
\label{eq:Sigma_T}
(\Sigma^Ff) \circ \Th(\varphi) = \Th(\varphi) \circ (\Sigma^Ef) \colon \Th_S(E) \to \Th_S(F).
\end{equation}
(This follows from the fact that, as morphisms $\Un_S \wedge \Th_S(E) \to  \Un_S \wedge \Th_S(F)$\[
(f \wedge \id_{\Th_S(F)}) \circ (\id_{\Un_S} \wedge \Th(\varphi)) = f \wedge \Th(\varphi) = (\id_{\Un_S} \wedge \Th(\varphi)) \circ (f \wedge \id_{\Th_S(E)}).)
\]
\end{para}

\begin{para}
\label{p:langle_rangle}
(See \cite[Lemma~6.3.4]{Morel-Intro_A1}.)
Let $X \in \Sm_S$ and $u\in H^0(X,\Gm)$. Consider the automorphism $u \id_1 \colon 1 \to 1$ of the trivial line bundle over $X$, and set
\[
\langle u \rangle = \Sigma^{-2,-1} \Th(u \id_1) \in \Aut_{\SH(S)}(\Sup X).
\]
It follows from \eqref{eq:T_mult} that
\begin{equation}
\label{eq:<>_mult}
\langle uv \rangle = \langle u \rangle \circ \langle v \rangle \quad \text{ for $u,v \in H^0(X,\Gm)$.}
\end{equation}
When $X=S$ and $A \in \SH(S)$, we will denote again by $\langle u \rangle \in \Aut_{\SH(S)}(A)$ the morphism
\[
A = \Un_S \wedge A \xrightarrow{\langle u \rangle \wedge \id_A} \Un_S \wedge A=A.
\]
If $f\colon A \to B$ is a morphism in $\SH(S)$, then
\begin{equation}
\label{eq:<>_commutes}
f \circ \langle u \rangle = \langle u \rangle \circ f.
\end{equation}
\end{para}

\begin{para}
\label{p:eta}
We denote by $\eta \colon \Ab^2 \smallsetminus \{0\} \to \Pp^1$ in $\Spcp(S)$ the map $(x,y) \mapsto [x:y]$, where $\Ab^2\smallsetminus \{0\}$ is pointed by $(1,1)$ and $\Pp^1$ by $[1:1]$.
\end{para}

\begin{para}
We will consider the categories $\Spcp(S)[\eta^{-1}]$ and $\Spt(S)[\eta^{-1}]$ obtained by monoidally inverting the map $\eta$ of \rref{p:eta}, which can be constructed as left Bousfield localisations, as discussed in \cite[\S6]{Bachmann-real}. Their respective homotopy categories will be denoted by $\Hop(S)[\eta^{-1}]$ and $\SH(S)[\eta^{-1}]$, and we will usually omit the mention of the localisation functors.

A spectrum $A \in \Spt(S)$ is called \emph{$\eta$-periodic} if the map
\begin{equation}
\label{eq:mult_eta}
A \wedge \Su (\Ab^2 \smallsetminus \{0\}) \xrightarrow{\id \wedge \Su \eta} A \wedge \Su \Pp^1
\end{equation}
is an isomorphism in $\SH(S)$. The full subcategory of such objects in $\Spt(S)$ can be identified $\Spt(S)[\eta^{-1}]$.
\end{para}

\begin{para}
\label{p:sharp}
Let $X \in \Sm_S$ with structural morphism $f\colon X \to S$. Then there are Quillen adjunctions
\[
f_\sharp \colon \Spcp(X) \leftrightarrows \Spcp(S) \colon f^* \quad ; \quad f_\sharp \colon \Spt(X) \leftrightarrows \Spt(S) \colon f^*.
\]
The functor $f^*$ is induced by base-change, while $f_\sharp$ arises from viewing a smooth $X$-scheme as a smooth $S$-scheme by composing with $f$. These induce Quillen adjunctions
\[
f_\sharp \colon \Spcp(S)[\eta^{-1}] \leftrightarrows \Spcp(S)[\eta^{-1}] \colon f^* \quad ; \quad f_\sharp \colon \Spt(X)[\eta^{-1}] \leftrightarrows \Spt(S)[\eta^{-1}] \colon f^*.
\]
We will also use the notation $f^*,f_{\sharp}$ for the derived functors on the respective homotopy categories.
\end{para}

\begin{para}
\label{p:hyperdescent}
(See e.g.\ \cite{DHI}.) Let $V \to X$ be a vector bundle with $X \in \Sm_S$, and $U_\alpha$ an open covering of $X$. Then the map
\[
\hocolim \Big( \cdots\mathrel{\substack{\textstyle\rightarrow\\[-0.6ex]
                      \textstyle\rightarrow \\[-0.6ex]
                      \textstyle\rightarrow}} \coprod_{\alpha,\beta} \Th_{U_\alpha \cap U_\beta}(V|_{U_\alpha \cap U_\beta}) \rightrightarrows \coprod_\alpha \Th_{U_\alpha}(V|_{U_\alpha})\Big) \to \Th_X(V)
\]
is a weak equivalence.
\end{para}

\section{Splitting \texorpdfstring{$\Gm$}{Gm}-torsors}
\numberwithin{theorem}{subsection}
\numberwithin{lemma}{subsection}
\numberwithin{proposition}{subsection}
\numberwithin{corollary}{subsection}
\numberwithin{example}{subsection}
\numberwithin{definition}{subsection}
\numberwithin{remark}{subsection}

\subsection{Local splitting}
In this section we consider the schemes $\Ab^1,\Gm,\Pp^1,\Ab^2\smallsetminus \{0\}$ as pointed motivic spaces, respectively via $1,1,[1:1],(1,1)$. We recall that $\T=\Ab^1/\Gm$. We have a chain of weak equivalences 
\begin{equation}
\label{eq:T_P1}
\T = \Ab^1/\Gm \xrightarrow{\sim} \Pp^1/\Ab^1 \xleftarrow{\sim} \Pp^1,
\end{equation}
where the first arrow is induced by the immersion $\Ab^1 \to \Pp^1, x \mapsto [x:1]$, and the quotient $\Pp^1/\Ab^1$ is taken with respect to the immersion $\Ab^1 \to \Pp^1, y \mapsto [1:y]$.\\

We first recall a well-known fact (see e.g.\ \cite[Lemma~6.2]{Ana-SL} for a stable version):
\begin{lemma}
\label{lemm:square}
Let $u \in H^0(S,\Gm)$. Then the morphism $\Th(u^2 \id_1) \colon \T \to \T$ (see \rref{p:Thom_funct}) coincides with the identity in $\Hop(S)$.
\end{lemma}
\begin{proof}
The endomorphism $\varphi \colon \Pp^1 \to \Pp^1$ given by $[x:y] \mapsto [u^2x:y]=[ux:u^{-1}y]$ is induced by the matrix
\[
A=
\begin{pmatrix}
u & 0\\
0 & u^{-1}\\
\end{pmatrix}.
\]
Since
\[
A
=
\begin{pmatrix}
1 & u\\
0 & 1\\
\end{pmatrix}
\begin{pmatrix}
1 & 0\\
-u^{-1} & 1\\
\end{pmatrix}
\begin{pmatrix}
1 & u-1\\
0 & 1\\
\end{pmatrix}
\begin{pmatrix}
1 & 0\\
1 & 1\\
\end{pmatrix}
\begin{pmatrix}
1 & -1\\
0 & 1\\
\end{pmatrix}
\]
is a product of transvections, the endomorphism $\varphi$ induces the identity endomorphism of $\Pp^1_+$ in $\Hop(S)$ (see e.g.\ \cite[Lemma~1]{Ana-Pushforwards}). The map $\varphi$ stabilises the copies of $\Ab^1$ given by $x \mapsto [x:1]$ and $y \mapsto [1:y]$, and restricts to $u^2 \id_1$ on the former. Thus the statement follows from the isomorphism \eqref{eq:T_P1}.
\end{proof}

\begin{para}
Excision yields isomorphisms in $\Hop(S)$
\[
(\Ab^2 \smallsetminus \{0\})/(\Gm \times \Ab^1) \xleftarrow{\sim} (\Ab^1 \times \Gm)/(\Gm \times \Gm) \xrightarrow{\sim} (\Ab^1/\Gm) \wedge (\Gm)_+ = \T \wedge (\Gm)_+
\]
Composing with the quotient $\Ab^2 \smallsetminus \{0\} \to (\Ab^2 \smallsetminus \{0\})/(\Ab^1 \times \Gm)$, this yields a map
\begin{equation}
\label{eq:theta}
\Ab^2 \smallsetminus \{0\} \to \T \wedge (\Gm)_+.
\end{equation}
\end{para}

\begin{lemma}
\label{lemm:cofiber}
The projection $p\colon \Gm \to S$ induces a cofiber sequence in $\Spcp(S)$
\[
\Ab^2 \smallsetminus \{0\} \xrightarrow{\eqref{eq:theta}} \T \wedge (\Gm)_+ \xrightarrow{\id \wedge p_+} \T
\]
\end{lemma}
\begin{proof}
This follows from the consideration of the following commutative diagram in $\Spcp(S)$, whose rows are cofiber sequences
\[ \xymatrix{
\Ab^2 \smallsetminus \{0\} \ar[r]  & (\Ab^2 \smallsetminus \{0\})/(\Gm \times \Ab^1) \ar[r] & S^{1,0}\wedge (\Gm \times \Ab^1) \ar@/^6.0pc/[ddd]^{\sim}\\ 
\Ab^1 \times \Gm \ar[r] \ar[u]\ar[d] & (\Ab^1 \times \Gm)/(\Gm \times \Gm) \ar[u]_{\sim} \ar[d]^{\sim}\ar[r] & S^{1,0} \wedge (\Gm \times \Gm) \ar[u]\ar[d]\\
\Ab^1 \wedge (\Gm)_+ \ar[r] \ar[d]_{\id \wedge p_+}& (\Ab^1/\Gm) \wedge (\Gm)_+ \ar[r] \ar[d]_{\id \wedge p_+}& S^{1,0} \wedge \Gm \wedge (\Gm)_+ \ar[d]_{\id \wedge \id \wedge p_+}\\
\Ab^1 \ar[r] & \Ab^1/\Gm \ar[r]^{\sim} & S^{1,0} \wedge \Gm
}\]
and where the curved arrow is the weak equivalence induced by the projection $\Gm \times \Ab^1 \to \Gm$.
\end{proof}

The next lemma is reminiscent of \cite[Theorem~3.8]{Ana-MSL}:
\begin{lemma}
\label{lemm:eta}
The morphism $\eta$ of \rref{p:eta} factors in $\Hop(S)$ as
\[
\Ab^2 \smallsetminus \{0\} \xrightarrow{\eqref{eq:theta}} \T \wedge (\Gm)_+ \xrightarrow{\Th(t^{-1}\id_1)} \T \wedge (\Gm)_+ \xrightarrow{\id \wedge p_+} \T \xrightarrow{\eqref{eq:T_P1}} \Pp^1,
\]
where $t \in H^0(\Gm,\Gm)$ is the tautological section and $p\colon \Gm \to S$ the projection (and $\T \wedge (\Gm)_+$ is identified with $\Th_{\Gm}(1)$).
\end{lemma}
\begin{proof}
Consider the commutative diagram in $\Sm_S$
\[ \xymatrix{
\Ab^1 \times \Gm \ar[r] \ar[d]_{\mu} & \Ab^2 \smallsetminus \{0\} \ar[d]^{\eta} \\ 
\Ab^1 \ar[r] & \Pp^1
}\]
where the upper horizontal arrow is the natural open immersion, the lower horizontal arrow is given by $x \mapsto [x:1]$, and $\mu$ is given by $(x,y)  \mapsto xy^{-1}$. Excision yields the isomorphisms in the commutative diagram in $\Hop(S)$
\[ \xymatrix{
(\Ab^1 \times \Gm)/(\Gm \times \Gm)\ar[rr]^-{\sim} \ar[d]_{\tilde{\mu}} && (\Ab^2 \smallsetminus \{0\})/(\Gm \times \Ab^1) \ar[d] & \Ab^2 \smallsetminus \{0\} \ar[d]^{\eta} \ar[l]\\ 
\Ab^1/\Gm \ar[rr]^{\sim} && \Pp^1/\Ab^1 & \Pp^1 \ar[l]_{\sim}
}\]
where $\tilde{\mu}$ is induced by $\mu$, and the lower horizontal arrows are the morphisms of \eqref{eq:T_P1}. To conclude, observe that the morphism 
 $\tilde{\mu}$ factors as the upper horizontal composite in the following commutative diagram in $\Hop(S)$:
\[ 
\begin{gathered}[b]
\xymatrix{
(\Ab^1 \times \Gm)/(\Gm \times \Gm) \ar[rrr]^{(x,y) \mapsto (xy^{-1},y)} \ar[d]^{\sim} &&& (\Ab^1 \times \Gm)/(\Gm \times \Gm) \ar[rr]^-{(x,y) \mapsto x} \ar[d]^{\sim} && \Ab^1/\Gm \ar[d]^{=}\\ 
\T \wedge (\Gm)_+ \ar[rrr]^-{\Th(t^{-1}\id_1)} &&& \T \wedge (\Gm)_+ \ar[rr]^-{\id \wedge p_+} && \T
}
\end{gathered}
\qedhere
\]
\end{proof}

\begin{proposition}
\label{prop:splitting_Gm}
Let $p \colon \Gm \to S$ be the projection, and $\T = \Ab^1/\Gm=\Th_S(1)$. Consider the composite (see \rref{p:Thom_funct})
\[
\pi \colon \T \wedge (\Gm)_+ \xrightarrow{\Th(t \id_1)} \T \wedge (\Gm)_+ \xrightarrow{\id \wedge p_+} \T.
\]
Then the following square is homotopy cocartesian in $\Spcp(S)[\eta^{-1}]$
\[ \xymatrix{
\T \wedge (\Gm)_+\ar[d]_-{\id \wedge p_+} \ar[r]^-{\pi} & \T \ar[d] \\ 
\T \ar[r] & {*}
}\]
\end{proposition}
\begin{proof}
From \rref{lemm:cofiber} we deduce a commutative diagram
\begin{equation}
\label{diag:cocartesian}
\begin{gathered}
\xymatrix{
\Ab^2 \smallsetminus \{0\} \ar[rr]^-{\eqref{eq:theta}} \ar[d] && \T \wedge (\Gm)_+ \ar[d]_{\id \wedge p_+} \ar[r]^-{\pi} & \T \ar[d] \\ 
{*} \ar[rr] && \T \ar[r] & {*}
}
\end{gathered}
\end{equation}
where the left inner square is homotopy cocartesian. Applying \rref{lemm:square} over the base $\Gm$ and using the functor $p_\sharp$ of \rref{p:sharp}, we have in $\Hop(S)$
\[
\Th(t \id_1) = \Th(t^{-1} \id_1) \colon \T \wedge (\Gm)_+ \to \T \wedge (\Gm)_+.
\]
It thus follows from \rref{lemm:eta} that the upper composite in the diagram \eqref{diag:cocartesian} is an isomorphism in $\Hop(S)[\eta^{-1}]$, hence the exterior square in \eqref{diag:cocartesian} is homotopy cocartesian. We conclude that the right inner square is homotopy cocartesian (by \cite[Proposition~13.3.15, Remark~7.1.10]{Hirschhorn}).
\end{proof}

\subsection{Global splitting}
\begin{definition}
\label{def:pi}
Let $L\to S$ be a line bundle. Denote by $L^\circ$ the complement of the zero-section in $L$, and by $p\colon L^{\circ} \to S$ the projection. Then the graph $L^\circ \to L^\circ \times_S L$ of the open immersion $L^\circ \to L$ may be viewed as a nowhere vanishing section of the line bundle $p^*L$ over $L^\circ$, which induces the \emph{tautological trivialisation} $\tau \colon 1 \xrightarrow{\sim} p^*L$. We define a morphism in $\Hop(S)$ (recall that $\T=\Th_S(1)$, and see \rref{p:Thom_funct})
\[
\pi_L \colon \T \wedge (L^\circ)_+ = \Th_{L^\circ}(1) \xrightarrow{\Th(\tau)} \Th_{L^\circ}(p^*L) \xrightarrow{p} \Th_S(L).
\]
\end{definition}

\begin{example}
\label{ex:pi_1}
Assume that $L=1$. Then the tautological trivialisation $\tau \colon 1 \to 1$ of the trivial line bundle over $L^\circ = \Gm$ is given by multiplication by the tautological section $t \in H^0(\Gm,\Gm)$, hence $\pi_1$ coincides with morphism $\pi$ of \rref{prop:splitting_Gm}.
\end{example}

\begin{para}
\label{p:pi_pb}
Let $L\to S$ be a line bundle. If $f \colon R \to S$ is a scheme morphism, then the functor $f^* \colon \Spcp(S) \to \Spcp(R)$ maps $\pi_L$ to $\pi_{f^*L}$.
\end{para}

\begin{para}
\label{p:pi_isom}
If $\varphi \colon L \xrightarrow{\sim} M$ is an isomorphism of line bundles over $S$, then the following diagram commutes in $\Hop(S)$
\[ \xymatrix{
\T \wedge (L^\circ)_+\ar[rr]^-{\pi_L} \ar[d]_{\id \wedge (\varphi^\circ)_+} && \Th_S(L) \ar[d]^{\Th(\varphi)} \\ 
\T \wedge (M^\circ)_+  \ar[rr]^-{\pi_M} && \Th_S(M)
}\]
\end{para}

\begin{proposition}
\label{prop:L_circ}
Let $L \to S$ be a line bundle. Then the following square is homotopy cocartesian in $\Spcp(S)[\eta^{-1}]$
\[ \xymatrix{
\T \wedge (L^\circ)_+\ar[d]_-{\id \wedge p_+} \ar[r]^-{\pi_L} & \Th_S(L) \ar[d] \\ 
\T \ar[r] & {*}
}\]
\end{proposition}
\begin{proof}
Let $F$ be the homotopy colimit of the diagram $\Th_S(L) \xleftarrow{\pi_L} \T \wedge (L^\circ)_+ \xrightarrow{\id \wedge p_+} \T$. By \rref{p:hyperdescent} and \rref{p:sharp} (and in view of \rref{p:pi_pb}), the fact that 
$F \simeq *$ may be verified Zariski-locally on $S$. We may thus assume that $L$ is trivial. By \rref{p:pi_isom} we may further assume that $L=1$, so that $L^\circ = \Gm$. Then, in view of \rref{ex:pi_1} the result follows from \rref{prop:splitting_Gm}
\end{proof}

The next statement was initially inspired by \cite[Proof of Theorem~4.1]{Levine-motivic_Euler}:

\begin{corollary}
\label{cor:L_circ}
Let $L \to S$ be a line bundle. Then in the notation of \rref{def:pi}, we have an isomorphism in $\SH(S)[\eta^{-1}]$
\[
(\Sup p,\Sigma^{-2,-1}\Su \pi_L) \colon \Sup L^{\circ} \xrightarrow{\sim} \Un_S \vo \Sigma^{-2,-1}\Th_S(L).
\]
\end{corollary}
\begin{proof}
The square induced in $\Spt(S)[\eta^{-1}]$ by the square of \rref{prop:L_circ}
 is homotopy cocartesian, hence also homotopy cartesian (see e.g.\ \cite[Remark~7.1.12]{Hovey-Model}). This yields an isomorphism
\[
(\Su (\id \wedge p_+),\Su \pi_L) \colon \Su (\T \wedge (L^{\circ})_+) \xrightarrow{\sim} \Su \T \vo \Th_S(L),
\]
from which the result follows by applying the functor $\Sigma^{-2,-1}$.
\end{proof}

\begin{corollary}
\label{cor:split_line}
Let $L\to S$ be a line bundle, and $V \to S$ a vector bundle.
\begin{enumerate}[(i)]
\item
\label{cor:split_line:1}
The natural map $\Th_{L^\circ}(V) \to \Th_S(V)$ extends to a natural isomorphism in $\SH(S)[\eta^{-1}]$
\[
\Th_{L^{\circ}}(V) \simeq \Th_S(V) \vo \Sigma^{-2,-1}\Th_S(V \oplus L).
\]
\item
\label{cor:split_line:2}
Denote by $p \colon L^\circ \to S$ and  $p_1,p_2 \colon L^\circ \times_S L^\circ \to L^\circ$ the projections. Then
\[
\Th_{L^\circ \times_S L^\circ}(V) \overset{p_1}{\underset{p_2}\rightrightarrows} \Th_{L^\circ}(V) \xrightarrow{p} \Th_S(V)
\]
is a split coequaliser diagram in $\SH(S)[\eta^{-1}]$.
\end{enumerate}
\end{corollary}
\begin{proof}
Statement \eqref{cor:split_line:1} follows by applying the auto-equivalence $\Sigma^V \colon \SH(S)[\eta^{-1}] \to \SH(S)[\eta^{-1}]$ to the decomposition of \rref{cor:L_circ}, in view of \eqref{eq:Thom_sum}.

Certainly in the diagram of \eqref{cor:split_line:2} we have $p \circ p_1 = p \circ p_2$. The isomorphism \eqref{cor:split_line:1} yields a section $s \colon \Th_S(V) \to \Th_{L^\circ}(V)$ of $p$ in $\SH(S)[\eta^{-1}]$. Then, in $\SH(S)[\eta^{-1}]$, the composite
\[
t \colon \Th_{L^\circ}(V) = (\Sup L^\circ) \wedge \Th_S(V) \xrightarrow{\id \wedge s} (\Sup L^\circ) \wedge \Th_{L^\circ}(V) = \Th_{L^\circ \times_S L^\circ}(V)
\]
is a section of
\[
p_1 \colon  \Th_{L^\circ \times_S L^\circ}(V) = (\Sup L^\circ) \wedge \Th_{L^\circ}(V) \xrightarrow{\id \wedge p} (\Sup L^\circ) \wedge \Th_S(V)  = \Th_{L^\circ}(V).
\]
On the other hand, in the commutative diagram in $\SH(S)[\eta^{-1}]$
\[ \xymatrix{
\Th_{L^\circ}(V)\ar@{=}[r] \ar[d]_p &  (\Sup L^\circ) \wedge \Th_S(V) \ar[r]^-{\id \wedge s}\ar[d]_{p \wedge \id}&(\Sup L^\circ) \wedge \Th_{L^\circ}(V) \ar@{=}[r] \ar[d]^{p \wedge \id} & \Th_{L^\circ \times_S L^\circ}(V)\ar[d]^{p_2} \\ 
\Th_S(V) \ar@{=}[r] & \Un_S \wedge \Th_S(V) \ar[r]^-{\id \wedge s} &  \Un_S \wedge \Th_{L^\circ}(V) \ar@{=}[r]& \Th_{L^\circ}(V)
}\]
the upper composite is $t$, while the lower one is $s$. Therefore $p_2 \circ t = s \circ p$ as endomorphisms of $\Th_{L^\circ}(V)$ in $\SH(S)[\eta^{-1}]$, proving  \eqref{cor:split_line:2}.
\end{proof}

\begin{corollary}
\label{cor:L_circ_faithful}
Let $L \to S$ be a line bundle, and denote by $p \colon L^\circ \to S$ the projection. Then the functor $p^* \colon \SH(S)[\eta^{-1}] \to \SH(L^\circ)[\eta^{-1}]$ is faithful and conservative.
\end{corollary}
\begin{proof}
By the smooth projection formula and \rref{cor:L_circ}, the composite $p_\sharp \circ p^* \colon \SH(S)[\eta^{-1}] \to \SH(S)[\eta^{-1}]$ decomposes as
\[
p_\sharp \circ p^* = \id \wedge (\Sup L^\circ) = \id \wedge (\Un_S \vo \Sigma^{-2,-1}\Th_S(L)) = \id \vo (\Sigma^{-2,-1} \circ \Sigma^L),
\]
which is faithful, hence so is $p^*$. The above formula also shows that $p_\sharp \circ p^*$ reflects zero-objects, hence so does $p^*$. Since $p^*$ is triangulated, it is conservative.
\end{proof}

\begin{remark}
The results of this section on line bundles will be generalised to odd rank vector bundles in \S\ref{sect:odd_rank}.
\end{remark}

\section{Applications to twisted cohomology}
\numberwithin{theorem}{subsection}
\numberwithin{lemma}{subsection}
\numberwithin{proposition}{subsection}
\numberwithin{corollary}{subsection}
\numberwithin{example}{subsection}
\numberwithin{definition}{subsection}
\numberwithin{remark}{subsection}

\subsection{Cohomology theories represented by ring spectra}
\begin{para}
\label{p:coh}
Let $A \in \Spt(S)$ be a motivic spectrum. For a pointed motivic space $\mathcal{X}$ we write
\[
A^{p,q}(\mathcal{X}) = \Hom_{\SH(S)}(\Sigma^\infty\mathcal{X},\Sigma^{p,q}A),
\]
and $\Aa(\mathcal{X}) = \bigoplus_{p,q\in\Zz} A^{p,q}(\mathcal{X})$. When $X$ is a smooth $S$-scheme, we will write $\Aa(X)$ instead of $\Aa(X_+)$. If $E \to X$ is a vector bundle of constant rank $r$, we write
\[
A^{p,q}(X;E) = A^{p+2r,q+r}(\Th_X(E)),
\]
and extend this notation to arbitrary vector bundles in an obvious way. A morphism $f \colon \mathcal{Y} \to \mathcal{X}$ of pointed motivic spaces (resp.\ of smooth $S$-schemes) induces a pullback $f^* \colon \Aa(\mathcal{X}) \to \Aa(\mathcal{Y})$.
\end{para}

\begin{para}
A commutative ring spectrum will mean a commutative monoid in $(\SH(S),\wedge,\Un_S)$. When $A \in \SH(S)$ is a commutative ring spectrum and $X \in \Sm_S$, then $\Aa(X)$ is naturally a ring, and $\Aa(X;E)$ an $\Aa(X)$-module. When $u \in H^0(X,\Gm)$, we will write $\langle u \rangle \in A^{0,0}(X)$ instead of $\langle u \rangle^*(1)$ (see \rref{p:langle_rangle}).
\end{para}

\begin{para}
If $A$ is an $\eta$-periodic motivic spectrum, for any pointed motivic space $\mathcal{X}$, we have natural isomorphisms for $p,q\in \Zz$
\[
A^{p,q}(\mathcal{X}) = \Hom_{\SH(S)[\eta^{-1}]}(\Su \mathcal{X},\Sigma^{p,q}A).
\]
\end{para}

\begin{proposition}
\label{cor:split_odd:coh}
Let $A$ be an $\eta$-periodic motivic spectrum. Let $X \in \Sm_S$. Let $L \to X$ be a line bundle, and $V \to X$ a vector bundle.
\begin{enumerate}[(i)]
\item
\label{cor:split_odd:coh:1}
Denoting by $p \colon L^\circ \to X$ the projection, we have a split short exact sequence
\[
0 \to \Aa(X;V) \xrightarrow{p^*} \Aa(L^\circ;V) \to \Aa(X;V \oplus L) \to 0.
\]

\item 
\label{cor:split_odd:coh:2}
Denoting by $p_1,p_2 \colon L^\circ \times_X L^\circ \to L^\circ$ the projections, we have an exact sequence
\[
0 \to \Aa(X;V) \xrightarrow{p^*} \Aa(L^\circ;V) \xrightarrow{p_1^*-p_2^*} \Aa(L^\circ \times_X L^\circ;V).
\]
\end{enumerate}
\end{proposition}
\begin{proof}
This follows by applying \rref{cor:split_line} over the base $X$ to the image of $A$ under the pullback $\Spt(S) \to \Spt(X)$.
\end{proof}

\subsection{\texorpdfstring{$\SL$}{SL}- and \texorpdfstring{$\SL^c$}{SLc}-orientations}

\begin{definition}
An \emph{$\SL$-oriented vector bundle} over a scheme $X$ is a pair $(E,\delta)$, where $E \to X$ is a vector bundle and $\delta \colon 1 \xrightarrow{\sim} \det E$ is an isomorphism of line bundles. We will also say that $\delta$ is an \emph{$\SL$-orientation} of the vector bundle $E \to X$. An isomorphism of $\SL$-oriented vector bundles $(E,\delta) \xrightarrow{\sim} (F,\epsilon)$ is an isomorphism of vector bundles $\varphi \colon E \xrightarrow{\sim} F$ such that $(\det \varphi) \circ \delta= \epsilon$.
\end{definition}

\begin{definition}[See {\cite[\S3]{PW-BO}}]
An \emph{$\SL^c$-oriented vector bundle} over a scheme $X$ is a triple $(E,L,\lambda)$, where $E \to X$ is a vector bundle and $L \to X$ a line bundle, and $\lambda \colon L^{\otimes 2} \xrightarrow{\sim} \det E$ is an isomorphism. We will also say that $(L,\lambda)$ is an \emph{$\SL^c$-orientation} of the vector bundle $E \to X$. An isomorphism of $\SL^c$-oriented vector bundles $(E,L,\lambda) \xrightarrow{\sim} (F,M,\mu)$ is an isomorphism of vector bundles $\varphi \colon E \xrightarrow{\sim} F$ and an isomorphism of line bundles $\psi \colon L \xrightarrow{\sim} M$ such that $(\det \varphi) \circ \lambda= \mu \circ \psi^{\otimes 2}$.
\end{definition}

\begin{para}
\label{p:SL_SLc}
Observe that each $\SL$-orientation $\delta$ of a vector bundle $E$ induces an $\SL^c$-orientation $(L,\lambda)$ of $E$, where $L=1$ and $\lambda$ is the composite $1^{\otimes 2} \simeq 1 \xrightarrow{\delta} \det E$.
\end{para}

\begin{para}
\label{p:SLc_SL}
Let $(E,L,\lambda)$ be an $\SL^c$-oriented vector bundle, and assume that the line bundle $L$ is trivial. Then every trivialisation $\alpha \colon 1 \xrightarrow{\sim} L$ induces an $\SL$-orientation of $E$ given by
\[
\delta_\alpha \colon 1 \simeq 1^{\otimes 2} \xrightarrow{\alpha^{\otimes 2}} L^{\otimes 2} \xrightarrow{\lambda} \det E.
\]
Observe that the $\SL^c$-oriented vector bundle induced (in the sense of \rref{p:SL_SLc}) by $\delta_\alpha$ is isomorphic to $(E,L,\lambda)$.
\end{para}

\begin{para}
\label{def:SL-orientation}
Consider a commutative ring spectrum $A \in \SH(S)$. By a $\SL$-, resp.\ $\SL^c$-, \emph{orientation} of $A$, we will mean a normalised orientation in the sense of \cite[Definition~3.3]{Ana-SL}. Such data consists in Thom classes $\thom_{(E,\delta)} \in \Aa(X;E)$ for each $\SL$-oriented vector bundle $(E,\delta)$ over $X \in \Sm_S$, resp.\ $\thom_{(E,L,\lambda)} \in \Aa(X;E)$ for each $\SL^c$-oriented vector bundle $(E,L,\lambda)$ over $X \in \Sm_S$, satisfying a series of axioms.
\end{para}

\begin{para}
\label{p:A_SLc_SL}
Let $A \in \SH(S)$ be a commutative ring spectrum. Then each $\SL^c$-orientation of $A$ induces an $\SL$-orientation of $A$, by letting the Thom class of an $\SL$-oriented vector bundle be the Thom class of the induced $\SL^c$-oriented vector bundle, in the sense of \rref{p:SL_SLc}.
\end{para}

\begin{lemma}
\label{lemm:indep_triv}
Let $A \in \SH(S)$ be an $\SL$-oriented commutative ring spectrum. Let $(E,L,\lambda)$ be an $\SL^c$-oriented vector bundle over $X \in \Sm_S$, and assume that the line bundle $L$ is trivial. Then, in the notation of \rref{p:SLc_SL}, the Thom class $\thom_{(E,\delta_\alpha)} \in \Aa(X;E)$ does not depend on the choice of the trivialisation $\alpha$ of $L$.
\end{lemma}
\begin{proof}
If $\alpha \colon  1 \xrightarrow{\sim} L$ is a trivialisation, then every trivialisation is of the form $u \alpha$ for some $u \in H^0(X,\Gm)$. In the notation of \rref{p:SLc_SL} we then have $\delta_{u \alpha} = u^2  \delta_{\alpha}$. By \cite[Lemma~7.3]{Ana-SL}, we have
\[
\thom_{(E,\delta_{u \alpha})} = \thom_{(E,u^2 \delta_{\alpha})} = \langle u^2 \rangle \thom_{(E,\delta_\alpha)} \in \Aa(X;E).
\]
Since $\langle u^2 \rangle=1 \in A^{0,0}(X)$ by \rref{lemm:square} (or \cite[Lemma~6.2]{Ana-SL}), the statement follows. 
\end{proof}

\begin{proposition}
\label{prop:SL_SLc}
Let $A \in \SH(S)$ be an $\eta$-periodic commutative ring spectrum. Then every $\SL$-orientation of $A$ is induced (in the sense of \rref{p:A_SLc_SL}) by a unique $\SL^c$-orientation.
\end{proposition}
\begin{proof}
We assume given an $\SL$-orientation of $A$. Let $(E,L,\lambda)$ be an $\SL^c$-oriented vector bundle. Denoting by $p\colon L^\circ \to X$ the projection, the line bundle $p^*L$ over $L^\circ$ admits a tautological trivialisation $\tau \colon 1 \xrightarrow{\sim} p^*L$. In view of \rref{p:SLc_SL}, this yields an $\SL$-orientation $\delta_\tau$ of $p^*E$.

First assume given an $\SL^c$-orientation of $A$ compatible with its $\SL$-orientation. As observed in \rref{p:SLc_SL}, the $\SL^c$-oriented vector bundle $(p^*E,p^*L,p^*\lambda)$ is isomorphic to the one induced by the $\SL$-oriented vector bundle $(p^*E,\delta_\tau)$. Thus we must have
\[
p^* \thom_{(E,L,\lambda)} = \thom_{(p^*E,p^*L,p^*\lambda)} = \thom_{(p^*E,\delta_\tau)} \in \Aa(L^{\circ};p^*E).
\]
Since $p^* \colon \Aa(X;E) \to \Aa(L^{\circ};p^*E)$ is injective by \rref{cor:split_odd:coh}, we obtain the uniqueness part of the statement.

We now construct an $\SL^c$-orientation of $A$ from its $\SL$-orientation. In the situation considered at the beginning of the proof, let $p_1,p_2 \colon L^\circ \times_X L^\circ \to L^\circ$ be the projections, and set $q=p \circ p_1=p \circ p_2$. The tautological trivialisation $\tau$ of $p^*L$ yields two trivialisations $p_1^*\tau$ and $p_2^*\tau$ of $q^*L$, and thus two $\SL$-orientations $\alpha_1=\delta_{p_1^*\tau}$ and $\alpha_2=\delta_{p_2^*\tau}$ of $E$. However it follows from \rref{lemm:indep_triv} that their Thom classes coincide, so that (observe that $\alpha_i = p_i^*(\delta_\tau)$ for $i=1,2$)
\[
p_1^*\thom_{(p^*E,\delta_\tau)} = \thom_{(E,\alpha_1)} = \thom_{(E,\alpha_2)} = p_2^*\thom_{(p^*E,\delta_\tau)} \in \Aa(L^{\circ} \times_X L^{\circ};q^*E).
\]
Therefore it follows from \dref{cor:split_odd:coh}{cor:split_odd:coh:2} that the element $\thom_{(p^*E,\delta_\tau)} \in \Aa(L^\circ;E)$ is the image of a unique element $\theta_{(E,L,\lambda)} \in \Aa(X;E)$.

From the fact that $(E,\delta) \mapsto \thom_{(E,\delta)}$ defines an $\SL$-orientation of $A$, we deduce at once that $(E,L,\lambda) \mapsto \theta_{(E,L,\lambda)}$ defines an $\SL^c$-orientation of $A$: indeed, each axiom of \cite[Definition~3.3]{Ana-SL} can be verified after pulling back along $p\colon L^\circ \to X$, since $p^* \colon \Aa(X;E) \to \Aa(L^\circ;p^*E)$ is injective by \rref{cor:split_odd:coh}.

To conclude the proof, it remains to show the $\SL^c$-orientation $(E,L,\lambda) \mapsto \theta_{(E,L,\lambda)}$ induces the original $\SL$-orientation of $A$. So let us assume that the $\SL^c$-oriented vector bundle $(E,L,\lambda)$ is induced by an $\SL$-oriented vector bundle $(E,\delta)$, in the sense of \rref{p:SL_SLc}. In particular $L=1$. Then the tautological trivialisation $\tau \colon 1 \xrightarrow{\sim} p^*L$ and the trivialisation $1 =p^*1 =p^*L$ yield two $\SL$-orientations of $p^*E$. Their Thom classes in $\Aa(L^\circ;p^*E)$ coincide by \rref{lemm:indep_triv}, and they are respectively $p^*\theta_{(E,L,\lambda)}$ and $p^*\thom_{(E,\delta)}$. Since $p^* \colon \Aa(X;E) \to \Aa(L^\circ;p^*E)$ is injective by \rref{cor:split_odd:coh}, we have $\theta_{(E,L,\lambda)} = \thom_{(E,\delta)} \in \Aa(X;E)$, as required.
\end{proof}

\begin{remark}
Ananyevskiy constructed ``Thom isomorphisms'' associated with $\SL^c$-bundles in \cite[\S4]{Ana-SL} when $A$ is an arbitrary $\SL$-oriented theory, but as explained in \cite[Remark~4.4]{Ana-SL} it is not clear whether this yields an $\SL^c$-orientation, the problem being the multiplicativity axiom. When $A$ is $\eta$-periodic, our construction leads to the same Thom isomorphisms for $\SL^c$-bundles (in fact the proof of \rref{prop:SL_SLc} shows that there is at most one way to construct such functorial isomorphisms compatibly with the $\SL$-orientation). Thus the Thom isomorphisms constructed by Ananyevskiy do give rise to an $\SL^c$-orientation when $A$ is $\eta$-periodic.
\end{remark}

\subsection{Twisting by doubles and squares of line bundles}
\label{sect:tw_doubles_squares}

\begin{proposition}
\label{prop:double_square}
Let $L$ be a line bundle over $S$. 
\begin{enumerate}[(i)]
\item \label{prop:double_square:plus}
There exist an isomorphism $\Sigma^{4,2}\Un_S \simeq \Th_S(L^{\oplus 2})$ in $\SH(S)[\eta^{-1}]$.

\item \label{prop:double_square:tensor}
If $s_1,s_2 \in \Zz$ are of the same parity, there exist an isomorphism
\[
\Th_S(L^{\otimes s_1}) \simeq \Th_S(L^{\otimes s_2}) \; \text{in $\SH(S)[\eta^{-1}]$}.
\]
\end{enumerate}
\end{proposition}
\begin{proof}
Let us first assume that the line bundle $L \to S$ admits a trivialisation $\alpha \colon 1 \xrightarrow{\sim} L$. Then we have an isomorphism in $\SH(S)$ (see \rref{p:Thom_funct})
\begin{equation}
\label{eq:Th_alpha_oplus_2}
\Th(\alpha^{\oplus 2}) \colon \Th_S(1^{\oplus 2})  \xrightarrow{\sim} \Th_S(L^{\oplus 2}).
\end{equation}
Every trivialisation of $L$ is of the form $u \alpha$ with $u \in H^0(S,\Gm)$. As automorphisms of $\Th_S(1^{\oplus 2})$ in $\SH(S)$ we have,
\[
\Th((u\id_1)^{\oplus 2}) = \Th((u^2\id_1) \oplus \id_1) \circ \Th((u^{-1}\id_1) \oplus (u\id_1)) = \Th((u^2\id_1) \oplus \id_1)
\]
because $\Th((u^{-1}\id_1) \oplus (u\id_1))$ is the identity of $\Th_S(1^{\oplus 2})$, being given by a product transvections (see e.g.\ \cite[Lemma~1]{Ana-Pushforwards}). Now by \rref{lemm:square}, under the identification $\Th_S(1^{\oplus 2})=\Th_S(1) \wedge \Th_S(1)$ (see \eqref{eq:Thom_sum}), we have
\[
\Th((u^2\id_1) \oplus \id_1) = \Th(u^2\id_1) \wedge \id_{\Th_S(1)} = \id_{\Th_S(1)} \wedge \id_{\Th_S(1)} = \id_{\Th_S(1^{\oplus 2})}.
\]
Therefore $\Th((u\id_1)^{\oplus 2})$ is the identity of $\Th_S(1^{\oplus 2})$ in $\SH(S)$, hence
\[
\Th((u\alpha)^{\oplus 2}) = \Th(\alpha^{\oplus 2}) \circ \Th((u\id_1)^{\oplus 2}) = \Th(\alpha^{\oplus 2}),
\]
so that the isomorphism $\Th(\alpha^{\oplus 2})$ in $\SH(S)$ considered in \eqref{eq:Th_alpha_oplus_2} is independent of the choice of the trivialisation $\alpha$. 

Next let us consider the case \eqref{prop:double_square:tensor}. If $\alpha \colon 1 \xrightarrow{\sim} L$ is a trivialisation, we have an isomorphism in $\SH(S)$
\begin{equation}
\label{eq:Th_alpha_otimes}
\Th(\id_{L^{\otimes s_1}} \otimes \alpha^{\otimes s_2-s_1}) \colon \Th_S(L^{\otimes s_1})  \xrightarrow{\sim} \Th_S(L^{\otimes s_2}).
\end{equation}
(Here and below, for $r \in \Nn$, the notation $\alpha^{\otimes -r}$ refers to the morphism $((\alpha^\vee)^{-1})^{\otimes r}$.) Now for $u \in H^0(S,\Gm)$, the composite in $\SH(S)$
\[
\Th_S(1) \xrightarrow{\Th(\alpha^{\otimes s_1})} \Th_S(L^{\otimes s_1}) \xrightarrow{\Th(\id_{L^{\otimes s_1}} \otimes (u\alpha)^{\otimes s_2-s_1})} \Th_S(L^{\otimes s_2}) \xrightarrow{\Th(\alpha^{\otimes s_2})^{-1}} \Th_S(1)
\]
coincides with $\Th(u^{s_2-s_1} \id_1)$, which is the identity by \rref{lemm:square} (recall that $s_2-s_1$ is even), and in particular does not depend on $u\in H^0(S,\Gm)$. Since the left and right arrows in the above composite are isomorphisms, we deduce that the middle arrow does not depend on $u\in H^0(S,\Gm)$, which shows as above that the isomorphism \eqref{eq:Th_alpha_otimes} is independent of the choice of the trivialisation $\alpha$.

Let us come back to the general case, where $L \to S$ is a possibly nontrivial line bundle. Let $p\colon L^\circ \to S$ be the projection, and consider the  tautological trivialisation $\tau$ of the line bundle $p^*L$ over $L^\circ$. Let us consider the isomorphism $\varphi \colon \Th_{L^\circ}(B) \xrightarrow{\sim} \Th_{L^\circ}(C)$ in $\SH(S)[\eta^{-1}]$, where 
\begin{itemize}
\item $B = 1^{\oplus 2}, C=L^{\oplus 2}, \varphi = \Th(\tau^{\oplus 2})$ in case \eqref{prop:double_square:plus}.
\item $B = L^{\otimes s_1}, C=L^{\otimes s_2}, \varphi = \Th(\id_{L^{\otimes s_1}} \otimes \alpha^{\otimes s_2-s_1})$ in case \eqref{prop:double_square:tensor},
\end{itemize}
Let $p_1,p_2 \colon L^\circ \times_S L^\circ \to L^\circ$ be the projections, and set $q=p \circ p_1=p \circ p_2$. For $i\in \{1,2\}$, the isomorphism $p_i^*\varphi \colon \Th_{L^\circ \times_S L^\circ}(B) \to \Th_{L^\circ \times_S L^\circ}(C)$ in $\SH(S)[\eta^{-1}]$ is induced by the trivialisation $p_i^*\tau$ of the line bundle $q^*L$ over $L^\circ \times_S L^\circ$, hence does not depend on $i$, by the special case considered at the beginning of the proof. Thus, by \dref{cor:split_line}{cor:split_line:2} there exists a unique morphism $f$ fitting into the commutative diagram in $\SH(S)[\eta^{-1}]$
\[ \xymatrix{
\Th_{L^\circ \times_S L^\circ}(B)\ar@<-0.5ex>[r]_-{p_2} \ar@<+0.5ex>[r]^-{p_1}\ar[d]_{p_1^*\varphi = p_2^*\varphi} & \Th_{L^{\circ}}(B) \ar[r] \ar[d]^{\varphi} \ar[r]^-{p} & \Th_S(B) \ar[d]^f\\ 
\Th_{L^\circ \times_S L^\circ}(C) \ar@<-0.5ex>[r]_-{p_2} \ar@<+0.5ex>[r]^-{p_1} & \Th_{L^{\circ}}(C) \ar[r]_-{p}& \Th_S(C)
}\]
as well as a unique morphism $g$ into the commutative diagram in $\SH(S)[\eta^{-1}]$
\[ \xymatrix{
\Th_{L^\circ \times_S L^\circ}(C)\ar@<-0.5ex>[r]_-{p_2} \ar@<+0.5ex>[r]^-{p_1}\ar[d]_{p_1^*\varphi^{-1} = p_2^*\varphi^{-1}} & \Th_{L^{\circ}}(C) \ar[r] \ar[d]^{\varphi^{-1}} \ar[r]^-{p} & \Th_S(C) \ar[d]^g\\ 
\Th_{L^\circ \times_S L^\circ}(B) \ar@<-0.5ex>[r]_-{p_2} \ar@<+0.5ex>[r]^-{p_1} & \Th_{L^{\circ}}(B) \ar[r]_-{p}& \Th_S(B)
}\]
As  $p \colon \Th_{L^{\circ}}(B) \to \Th_S(B)$ and $p \colon \Th_{L^{\circ}}(C) \to \Th_S(C)$ are epimorphisms in $\SH(S)[\eta^{-1}]$ (see \rref{cor:split_line}), it follows that $f$ and $g$ are mutually inverse isomorphisms in $\SH(S)[\eta^{-1}]$.
\end{proof}

\begin{remark}
Proposition~\rref{prop:double_square} will be improved in \rref{prop:twist_square}.
\end{remark}

\section{Nowhere vanishing sections of odd rank bundles}
\numberwithin{theorem}{subsection}
\numberwithin{lemma}{subsection}
\numberwithin{proposition}{subsection}
\numberwithin{corollary}{subsection}
\numberwithin{example}{subsection}
\numberwithin{definition}{subsection}
\numberwithin{remark}{subsection}

\subsection{Projective bundles}

The results of this section are slight generalisations of those of \cite[\S4]{Ana-Pushforwards}.

\begin{para}
\label{p:Pk}
Let us consider the linear embeddings $i_k \colon \Pp^k \to \Pp^{k+1}$ given by the vanishing of the last coordinate. Denote by $\iota_k \colon S \to \Pp^k$ the $S$-point given by the composite $S=\Pp^0 \xrightarrow{i_0} \Pp^1 \xrightarrow{i_1} \cdots \xrightarrow{i_{k-1}} \Pp^k$. 
\end{para}

\begin{para}
\label{p:O_D}
Assume given a collection $\Dc=(d_1,\ldots,d_r) \in \Zz^r$ for some $r\in \Nn$. We will denote by $\Oc(\Dc)$ the vector bundle $\Oc(d_1) \oplus \cdots \oplus \Oc(d_r)$ over $\Pp^k$, for each $k \in \Nn$. When $k=0$, we have a canonical isomorphism $\Oc(\Dc) \simeq 1^{\oplus r}$ over $\Pp^0=S$. This yields, for any $k\in \Nn$, a canonical map in $\Spt(S)$
\begin{equation}
\label{eq:Sigma_triv_P}
\Sigma^{2r,r} \Un_S = \Th_S(1^{\oplus r}) \simeq \Th_{\Pp^0}(\Oc(\Dc)) \xrightarrow{\iota_k} \Th_{\Pp^k}(\Oc(\Dc)).
\end{equation}
\end{para}

\begin{proposition}
\label{prop:Th_Pk}
Let $k,r \in \Nn$, and $d_1,\ldots,d_r \in \Zz$. Set $\Dc=(d_1,\ldots,d_r)$, and let $d=d_1 +\cdots +d_r$. We use the notation $\Oc(\Dc)$ described in \rref{p:O_D}.
\begin{enumerate}[(i)]
\item \label{prop:Th_Pk:1}
If $k$ and $d$ are odd, then $\Th_{\Pp^k}(\Oc(\Dc)) = 0$ in $\SH(S)[\eta^{-1}]$.

\item \label{prop:Th_Pk:2}
If $k$ and $d$ are even, then \eqref{eq:Sigma_triv_P} induces an isomorphism $\Sigma^{2r,r}\Un_S \simeq \Th_{\Pp^k}(\Oc(\Dc))$ in $\SH(S)[\eta^{-1}]$.

\item \label{prop:Th_Pk:3}
If $k$ is even and $d$ is odd, then $\Th_{\Pp^k}(\Oc(\Dc)) \simeq \Sigma^{2(k+r),k+r}\Un_S$.

\item \label{prop:Th_Pk:4}
If $k$ is odd and $d$ is even, then $\Th_{\Pp^k}(\Oc(\Dc)) \simeq \Sigma^{2(k+r),k+r}\Un_S \oplus \Sigma^{2r,r}\Un_S$.
\end{enumerate}
\end{proposition}
\begin{proof}
Let us first prove (i). Assume that $k$ and $d$ are odd. Consider a linear embedding $\Pp^1 \to \Pp^k$. Its normal bundle is $\Oc(1)^{\oplus k-1}$, and its open complement is a vector bundle over $\Pp^{k-2}$. The corresponding zero-section $\Pp^{k-2} \to \Pp^k\smallsetminus \Pp^1$ induces an isomorphism in $\SH(S)$, and is the restriction of a linear embedding $\Pp^{k-2} \to \Pp^k$. Thus \rref{p:purity} yields a distinguished triangle in $\SH(S)$
\[
\Th_{\Pp^{k-2}}(\Oc(\Dc)) \to \Th_{\Pp^k}(\Oc(\Dc)) \to \Th_{\Pp^1}(\Oc(\Dc) \oplus \Oc(1)^{\oplus k-1}) \to \Sigma^{1,0} \Th_{\Pp^{k-2}}(\Oc(\Dc)).
\]
Using induction on the odd integer $k$, we are reduced to assuming that $k=1$. Now by \dref{prop:double_square}{prop:double_square:tensor} we have in $\SH(S)[\eta^{-1}]$
\[
\Th_{\Pp^1}(\Oc(\Dc)) \simeq \Sigma^{2s,s}(\Th_{\Pp^1}(\Oc(-1))^{\wedge r-s}),
\]
where $s$ is the number of indices $i \in \{1,\dots,r\}$ such that $d_i$ is even. Since $d$ is odd, so is $r-s$, and using \dref{prop:double_square}{prop:double_square:plus} we deduce that
\[
\Th_{\Pp^1}(\Oc(\Dc)) \simeq \Sigma^{2(r-1),r-1}\Th_{\Pp^1}(\Oc(-1)) \quad \text{in $\SH(S)[\eta^{-1}]$.}
\]
But $\Th_{\Pp^1}(\Oc(-1))$ vanishes in $\SH(S)[\eta^{-1}]$, because of the distinguished triangle (see \rref{eq:E_circ_dist})
\[
\Sup \Oc(-1)^{\circ} \to \Sup \Pp^1 \to \Th_{\Pp^1}(\Oc(-1)) \to \Sigma^{1,0}\Sup \Oc(-1)^{\circ}
\]
and the definition of the map $\eta$ (recall that $\Oc(-1)^\circ = \Ab^2 \smallsetminus \{0\}$). We have proved \eqref{prop:Th_Pk:1}.

Let us come back to the situation when $k$ and $d$ are arbitrary. Consider a linear embedding $\Pp^{k-1} \to \Pp^k$ avoiding the $S$-point $\iota_k \colon S \to \Pp^k$ (we write $\Pp^{-1}=\varnothing$). It is a closed immersion defined by the vanishing of a regular section of $\Oc(1)$. Its open complement is isomorphic to $\Ab^k$, and the morphism $j_k\colon S \to \Ab^k$ induced by $\iota_k$ induces an isomorphism in $\SH(S)$. The canonical trivialisation of $\Oc(\Dc)$ over $\Pp^0=S$ is the restriction along $j_k$ of a trivialisation of $\Oc(\Dc)|_{\Ab^k}$ (induced by the trivialisation of $\Oc(1)|_{\Ab^k}$ corresponding to the regular section of $\Oc(1)$ mentioned above). It follows that the map $\Th_{\Pp^0}(\Oc(\Dc)) \to \Th_{\Ab^k}(\Oc(\Dc))$ induced by $j_k$ induces an isomorphism in $\SH(S)$. Thus \rref{p:purity} yields a distinguished triangle in $\SH(S)$
\[
\Sigma^{2r,r}\Un_S \xrightarrow{\eqref{eq:Sigma_triv_P}} \Th_{\Pp^k}(\Oc(\Dc)) \to \Th_{\Pp^{k-1}}(\Oc(\Dc) \oplus \Oc(1)) \to \Sigma^{2r+1,r} \Un_S,
\]
so that \eqref{prop:Th_Pk:2} follows from \eqref{prop:Th_Pk:1}.

Consider now a linear embedding $s\colon S=\Pp^0 \to \Pp^k$ avoiding $i_{k-1}(\Pp^{k-1})$. Its open complement $U$ is a line bundle over $\Pp^{k-1}$. The corresponding zero-section $\Pp^{k-1} \to U$ induces an isomorphism in $\SH(S)$, and is the restriction of the linear embedding $i_{k-1} \colon \Pp^{k-1} \to \Pp^k$. Since the vector bundle $s^*\Oc(\Dc)$ and the normal bundle $s^*\Oc(1)^{\oplus k}$ to $s$ are both trivial, we have by \rref{p:purity} a distinguished triangle in $\SH(S)$
\begin{equation}
\label{eq:tr_k-1_k}
\Th_{\Pp^{k-1}}(\Oc(\Dc)) \to \Th_{\Pp^k}(\Oc(\Dc)) \to \Sigma^{2(k+r),k+r}\Un_S \to \Sigma^{1,0} \Th_{\Pp^{k-1}}(\Oc(\Dc)).
\end{equation}
Therefore \eqref{prop:Th_Pk:3} follows from \eqref{prop:Th_Pk:1}. 

Finally, assume that $k$ is odd and $d$ is even. It follows from \eqref{prop:Th_Pk:2} that the composite $\Th_{\Pp^{k-1}}(\Oc(\Dc)) \to \Th_{\Pp^k}(\Oc(\Dc)) \to \Th_{\Pp^{k+1}}(\Oc(\Dc))$ is an isomorphism in $\SH(S)[\eta^{-1}]$, hence $\Th_{\Pp^{k-1}}(\Oc(\Dc)) \to \Th_{\Pp^k}(\Oc(\Dc))$ admits a retraction, giving a splitting of the triangle \eqref{eq:tr_k-1_k}. In view of \eqref{prop:Th_Pk:2}, this proves \eqref{prop:Th_Pk:4}.
\end{proof}

\begin{corollary}
\label{cor:Pk_even}
If $k \in \Nn$ is even, the structural morphism $\Pp^k \to S$ induces an isomorphism $\Sup \Pp^k \xrightarrow{\sim} \Un_S$ in $\SH(S)[\eta^{-1}]$.
\end{corollary}
\begin{proof}
The structural morphism is retraction of $\iota_k$, so the corollary follows from \dref{prop:Th_Pk}{prop:Th_Pk:2} applied with $r=0$.
\end{proof}

\begin{proposition}
\label{lemm:Th_odd}
Let $E,V_1,\ldots,V_n$ be vector bundles of constant rank over $S$, and $d_1,\ldots,d_n \in \Zz$. Assume that $\rank E$ is even and that $d_1 \rank V_1 +\cdots+d_n \rank V_n$ is odd. Then
\[
\Th_{\Pp(E)}((\Oc(d_1)\otimes q^*V_1) \oplus \cdots \oplus (\Oc(d_n)\otimes q^* V_n))=0 \in \SH(S)[\eta^{-1}],
\]
where $q \colon \Pp(E) \to S$ is the projective bundle.
\end{proposition}
\begin{proof}
By \rref{p:hyperdescent} and \rref{p:sharp}, this may be verified Zariski-locally on $S$, so we may assume that $E,V_1,\ldots,V_n$ are all trivial. Then the statement follows from \dref{prop:Th_Pk}{prop:Th_Pk:1}.
\end{proof}

\begin{proposition}
\label{prop:PE_odd}
\label{cor:PE_odd}
Let $E,V$ be vector bundles over $S$. Assume that $E$ has constant odd rank. Then $\Th_{\Pp(E)}(V) \to \Th_S(V)$ is an isomorphism in $\SH(S)[\eta^{-1}]$. In particular $\Sup \Pp(E) \xrightarrow{\sim} \Un_S$ in $\SH(S)[\eta^{-1}]$.
\end{proposition}
\begin{proof}
By \rref{p:hyperdescent} and \rref{p:sharp}, this may be verified Zariski-locally on $S$, so we may assume that $E$ and $V$ are both trivial. Then the statement follows after suspending \rref{cor:Pk_even}.
\end{proof}

\subsection{Odd rank vector bundles}
\label{sect:odd_rank}
\begin{para}
\label{p:Th_PE_X}
Let $E\to S$ be a vector bundle. The composite $\Oc(-1) \subset E \times_S \Pp(E) \to E$ restricts to an isomorphism $\Oc(-1)^{\circ} \xrightarrow{\sim} E^\circ$, which is $\Gm$-equivariant.  We thus obtain a commutative diagram in $\Spt(S)$
\begin{equation}
\label{diag:O(1)_E}
\begin{gathered}
\xymatrix{
\Sup \Oc(-1)^\circ\ar[r] \ar[d]^{\simeq} & \Sup \Pp(E) \ar[d]\\ 
\Sup E^\circ \ar[r] & \Un_S
}
\end{gathered}
\end{equation}
which induces a morphism of the homotopy cofibers of the horizontal morphisms:
\begin{equation}
\label{eq:Th_PE_X}
\Th_{\Pp(E)}(\Oc(-1)) \to \Th_S(E) \in \SH(S).
\end{equation}
\end{para}

\begin{lemma}
\label{lemm:Th_PE_X}
Let $E \to S$ be a vector bundle of constant odd rank. Then \eqref{eq:Th_PE_X} is an isomorphism $\Th_{\Pp(E)}(\Oc(-1)) \xrightarrow{\sim} \Th_S(E)$ in $\SH(S)[\eta^{-1}]$.
\end{lemma}
\begin{proof}
It follows from \rref{cor:PE_odd} that both vertical arrows in the diagram \eqref{diag:O(1)_E} are isomorphisms in $\SH(S)[\eta^{-1}]$, hence so the induced map on homotopy cofibers.
\end{proof}

\begin{proposition}
\label{prop:split_odd}
Let $E \to S$ be a vector bundle of constant odd rank. Then the projection $E^\circ \to S$ admits a section in $\SH(S)[\eta^{-1}]$, inducing a decomposition in $\SH(S)[\eta^{-1}]$
\[
\Sup E^{\circ} \simeq \Un_S \vo \Sigma^{-2,-1}\Th_S(E).
\]
\end{proposition}
\begin{proof}
By \rref{cor:L_circ}, we have a splitting in $\SH(S)[\eta^{-1}]$
\[
\Sup \Oc(-1)^\circ \simeq \Sup \Pp(E) \oplus \Sigma^{-2,-1}\Th_{\Pp(E)}(\Oc(-1))
\]
and the statement follows from \rref{p:Th_PE_X}, \rref{cor:PE_odd} and \rref{lemm:Th_PE_X}.
\end{proof}

\begin{remark}
One may deduce that \rref{cor:split_line} and \rref{cor:L_circ_faithful} remain valid when $L$ is an odd rank vector bundle instead of a line bundle, using exactly the same arguments, but substituting \rref{prop:split_odd} for \rref{cor:L_circ}.
\end{remark}

We deduce the following splitting principle:
\begin{corollary}
Let $X \in \Sm_S$, and $E \to X$ be a vector bundle of constant odd rank. Then there exists a morphism $f\colon Y \to X$ in $\Sm_S$ whose image in $\SH(S)[\eta^{-1}]$ admits a section, and a vector bundle $F \to Y$ such that $f^*E \simeq F \oplus 1$.
\end{corollary}
\begin{proof}
Applying the functor $\SH(X)[\eta^{-1}] \to \SH(S)[\eta^{-1}]$ of \rref{p:sharp} we may assume that $X=S$. Let us denote by $p\colon (E^\vee)^\circ \to S$ the projection. Then $p^*E^\vee$ admits a nowhere vanishing section $s$. Its dual $s^\vee \colon p^*E \to 1$ is surjective. Letting $Q = \ker s^\vee$, we have an exact sequence of vector bundles over $(E^\vee)^\circ$
\begin{equation}
\label{eq:ses:1_E_Q}
0 \to Q \to p^*E \to 1 \to 0.
\end{equation}
Then we may find an affine bundle $g \colon Y \to (E^\vee)^\circ$ along which the pullback of the sequence \eqref{eq:ses:1_E_Q} splits (we may take for $Y$ the scheme parametrising the sections of $p^*E^\vee \to Q^\vee$, see e.g.\ \cite[p.243]{Riou}). Then $\Sup g \colon \Sup Y \to \Sup (E^\vee)^\circ$ is an isomorphism in $\SH(S)$, and $\Sup p\colon \Sup (E^\vee)^\circ \to \Un_S$ admits a section in $\SH(S)[\eta^{-1}]$ by \rref{prop:split_odd}. So we may set $f =p \circ g$.
\end{proof}

\subsection{Thom spaces of tensor products by line bundles}
\label{sect:tw_V}
We are now in position to slightly improve the result obtained in \S\ref{sect:tw_doubles_squares}.

\begin{lemma}
\label{lemm:T_u}
Let $E \to S$ be a vector bundle and $u \in H^0(S,\Gm)$. Assume that $E$ has constant rank $r$. Then in the notation of \rref{p:Thom_funct} and  \rref{p:langle_rangle}, we have in $\SH(S)[\eta^{-1}]$
\[
\Th(u \id_E) =\langle u^r \rangle \colon \Th_S(E) \to \Th_S(E).
\]
\end{lemma}
\begin{proof}
Since, under the identification $\Th_S(E\oplus 1) = \Th_S(E) \wedge \Th_S(1)$ (see \eqref{eq:Thom_sum}) we have
\[
\Th(u \id_{E\oplus 1})= \Th(u \id_E) \wedge \Th(u\id_1) =\Th(u \id_E) \wedge \langle u \rangle,
\]
we may replace $E$ with $E\oplus 1$ if necessary, and thus assume that $r$ is odd. The $\Gm$-equivariant isomorphism $E^\circ \simeq \Oc(-1)^\circ$ (see \rref{p:Th_PE_X}) yields a commutative square in $\SH(S)$
\[ \xymatrix{
\Th_{\Pp(E)}(\Oc(-1)) \ar[rr]^{\Th(u \id_{\Oc(-1)})} \ar[d] && \Th_{\Pp(E)}(\Oc(-1)) \ar[d] \\ 
\Th_S(E) \ar[rr]^{\Th(u \id_E)} && \Th_S(E)
}\]
where the vertical arrows coincide, and are isomorphisms in $\SH(S)[\eta^{-1}]$ by \rref{lemm:Th_PE_X}. In view of \eqref{eq:<>_commutes} we may replace $E \to S$ with $\Oc(-1) \to \Pp(E)$, and thus assume that $E$ is a line bundle. By \rref{cor:L_circ_faithful}, we may replace $S$ with $E^\circ$, and thus assume that the line bundle $E \to S$ admits a trivialisation $\alpha \colon 1 \xrightarrow{\sim} E$. Then we have a commutative square of isomorphisms in $\SH(S)$
\[ \xymatrix{
\Th_S(E) \ar[rr]^{\Th(u \id_E)}  && \Th_S(E) \\ 
\Th_S(1) \ar[rr]^{\Th(u \id_1)} \ar[u]^{\Th(\alpha)} && \Th_S(1) \ar[u]_{\Th(\alpha)}
}\]
By definition $\Th(u \id_1) = \langle u \rangle$, and we deduce using  \eqref{eq:<>_commutes} that $\Th(u \id_E) = \langle u \rangle  \in \Aut_{\SH(S)}(\Th_S(E))$. Since $\langle u^2 \rangle =\id$ by \rref{lemm:square} and $r$ is odd, it follows that $\langle u \rangle = \langle u^r \rangle$, concluding the proof.
\end{proof}

\begin{proposition}
\label{prop:twist_square}
Let $L \to S$ be a line bundle, and $V \to S$ a vector bundle of constant rank $r$. If $s \in \Zz$ is such that $rs$ is even, then there exists an isomorphism in $\SH(S)[\eta^{-1}]$
\[
\Th_S(V) \simeq \Th_S(V \otimes L^{\otimes s}).
\]
\end{proposition}
\begin{proof}
Upon replacing $V$ with $V \otimes L^{\otimes s}$, we may assume that $s \geq 0$. When $\alpha \colon 1 \xrightarrow{\sim}L$ is a trivialisation of the line bundle $L$ over $S$ we have an isomorphism in $\SH(S)$
\begin{equation}
\label{eq:Th_V_L}
\Th(\id_V \otimes \alpha^{\otimes s}) \colon \Th_S(V) \to \Th_S(V \otimes L^{\otimes s}).
\end{equation}
Any trivialisation of $L$ is of the form $u \alpha$ for some $u \in H^0(S,\Gm)$, and we have in $\SH(S)[\eta^{-1}]$, by \rref{lemm:T_u} and \rref{lemm:square} 
\[
\Th(\id_V \otimes (u \alpha)^{\otimes s}) = \Th(\id_V \otimes \alpha^{\otimes s}) \circ \Th(u^s\id_V) = \Th(\id_V \otimes \alpha^{\otimes s}) \circ \langle u^{rs} \rangle = \Th(\id_V \otimes \alpha^{\otimes s}).
\]
It follows that the image of the isomorphism \eqref{eq:Th_V_L} in $\SH(S)[\eta^{-1}]$ is independent of the choice of the trivialisation $\alpha$, and we conclude as in the proof of \rref{prop:double_square}.
\end{proof}

\section{Classifying spaces and characters}

\numberwithin{theorem}{subsection}
\numberwithin{lemma}{subsection}
\numberwithin{proposition}{subsection}
\numberwithin{corollary}{subsection}
\numberwithin{example}{subsection}
\numberwithin{definition}{subsection}
\numberwithin{remark}{subsection}

\subsection{Models for \'etale classifying spaces}
Here we recall some facts concerning the geometric models of the \'etale classifying space of a linear algebraic group given in \cite[\S4.2]{MV-A1}.

\begin{para}
\label{p:def_BG}
Let $G$ be a linear algebraic group over $S$. Let $(V_m,U_m,f_m)$, for $m\in \Nn \smallsetminus \{0\}$, be an admissible gadget with a nice (right) $G$-action, in the sense of \cite[Definition~4.2.1]{MV-A1}. Here $V_m \to S$ are $G$-equivariant vector bundles, and $U_m \subset V_m$ are $G$-invariant open subschemes where the $G$-action is free. Set $E_mG=U_m$ and $B_mG = (E_mG)/G$. Let us define $\BB G \in \Spc(S)$ as the colimit of the motivic spaces $B_m G$ as $m$ runs over $\Nn \smallsetminus \{0\}$. It is proved in \cite[Proposition~4.2.6]{MV-A1} that the weak-equivalence class of $\BB G$ does not depend on the choice of $(V_m,U_m,f_m)$. More precisely if $(V_m,U_m,f_m),(V'_m,U'_m,f'_m)$ are admissible gadgets with a nice $G$-action, and $U_m \to U'_m$ are $G$-equivariant morphisms commuting with the morphisms $f_m,f_m'$, then the induced morphism of motivic spaces $\colim_m (U_m/G) \to \colim_m (U'_m/G)$ is a weak equivalence. In the sequel we will refer to a system $(V_m,U_m,f_m)$ as above as a model for $\BB G$, and use the notation $E_mG,B_mG$.
\end{para}

\begin{para}
\label{p:colim_hocolim}
In the situation of \rref{p:def_BG}, since $B_1G$ is cofibrant and each $B_m G \to B_{m+1} G$ is a cofibration (for the model structure of \cite{MV-A1}), it follows that the colimit $\BB G$ is canonically weakly equivalent to the homotopy colimit of the motivic spaces $B_mG$ in $\Spc(S)$ (see e.g.\ \cite[Theorem~19.9.1]{Hirschhorn}).
\end{para}

\begin{para}
\label{p:pointed}
Let $G$ be a linear algebraic group over $S$, and choose a model for $\BB G$. Since the map $\colim_m E_mG \to S$ is a weak-equivalence of motivic spaces \cite[Proposition~4.2.3]{MV-A1}, we obtain a canonical morphism $S \to \BB G$ in $\Hot(S)$. We say that the model is \emph{pointed} if we are given an $S$-point of $E_1 G$. This yields map $S \to \BB G$ in $\Spc(S)$, whose image in $\Hot(S)$ is the canonical morphism described just above.
\end{para}

\begin{para}
\label{p:Totaro_model}
(See also \cite[p.133]{MV-A1}.) Let us fix an integer $n\in \Nn$ and describe an explicit model for $\BGL_n$. Fix an integer $p \geq n$ (we will typically take $p=n$). For $s \in \Nn$, we denote by $\Gr(n,s)$ the grassmannian of rank $n$ subbundles $U\subset 1^{\oplus s}$ over $S$ (for us a subbundle is locally split, so $1^{\oplus s}/U$ is a vector bundle). For each $m\in \Nn \smallsetminus \{0\}$, consider the $S$-scheme $V_{m,p}$ parametrising the vector bundles maps $1^{\oplus n} \to 1^{\oplus pm}$; then $V_{m,p} \to S$ is a vector bundle. Let $U_{m,p}$ the open subscheme of $V_{m,p}$ parametrising those vector bundle maps admitting Zariski-locally a retraction (i.e.\ making $1^{\oplus n}$ a subbundle of $1^{\oplus pm}$). Then the natural left $\GL_n$-action on $1^{\oplus n}$ induces a right $\GL_n$-action on $U_{m,p}$, which is free, and the quotient $U_{m,p}/\GL_n$ can be identified with the grassmannian $\Gr(n,pm)$. The inclusion $1^{\oplus m} \subset 1^{\oplus m+1}$ given by the vanishing of the last coordinate induces an inclusion
\[
1^{\oplus pm} = (1^{\oplus m})^{\oplus p} \subset (1^{\oplus m+1})^{\oplus p} = 1^{\oplus p(m+1)},
\]
which yields a $\GL_n$-equivariant morphism $f_{m,p} \colon U_{m,p} \to U_{m+1,p}$.

Then the family $(V_{m,p},U_{m,p},f_{m,p})$ is an admissible gadget with a nice $\GL_n$-action. Indeed the first condition of \cite[Definition~4.2.1]{MV-A1} is satisfied because $U_{1,p}$ possesses an $S$-point, and the second condition is satisfied with $j=2i$. The fact that the group $\GL_n$ is special implies the validity of condition (3) of \cite[Definition~4.2.4]{MV-A1}. We thus obtained a model for $\BGL_n$. We have just seen that this model is pointed (in the sense of \rref{p:pointed}); a canonical pointing when $p=n$ is induced by the identity of $1^{\oplus n}$.
\end{para}

\begin{para}
\label{p:B_closed}
Let $H \subset G$ be an inclusion of linear algebraic groups over $S$. Then any admissible gadget with a nice $G$-action is also one with a nice $H$-action (where the $H$-action is given by restricting the $G$-action). Indeed, the only non-immediate point is condition (3) of \cite[Definition~4.2.4]{MV-A1}. So let $F$ be a smooth $S$-scheme with a free right $H$-action. Consider the quotient $E=(F \times G)/H$, where the right $H$-action on $G$ is given by letting $h \in H$ act via $g \mapsto h^{-1}g$. Right multiplication in $G$ induces a free right $G$-action on $E$. For any $U \in \Sm_S$ with a right $G$-action we have isomorphisms
\[
(E \times U)/G \simeq ((F \times G)/H \times U)/G \simeq  (F \times (G\times U)/G)/H \simeq (F\times U)/H,
\]
which are functorial in $U$, and thus permit to identify the morphisms $(E \times U)/G \to E/G$ and $(F \times U)/H \to F/H$. Since the former is an epimorphism in the Nisnevich topology (as the group $G$ is nice), so is the latter.

Thus given a model for $\BB G$, we obtain a model for $\BB H$, where $E_mH=E_mG$ with the induced $H$-action. This yields morphisms
\[
B_m H = (E_mH)/H = (E_mG)/H \to (E_mG)/G = B_mG
\]
which are compatible with the transition maps as $m$ varies, and thus a map $\BB H \to \BB G$.
\end{para}

\begin{para}
\label{p:model_exists}
(See also \cite[Remark~4.2.7]{MV-A1}.)
Assume that $G$ is a linear algebraic group over $S$, and fix an embedding $G \subset \GL_n$ as a closed subgroup. By \rref{p:B_closed} every (pointed) model for $\BGL_n$ induces a (pointed) model for $\BB G$. Since $\BGL_n$ admits a pointed model by \rref{p:Totaro_model}, so does $\BB G$.
\end{para}

\subsection{Products}
\begin{para}
Let $G,G'$ be linear algebraic groups over $S$. Choose admissible gadgets $(V_m,U_m,f_m)$ with a nice $G$-action, and $(V'_m,U'_m,f'_m)$ with a nice $G'$-action (in the sense of \cite[Definition~4.2.1]{MV-A1}, recall from \rref{p:model_exists} that such exist). Then the family $(V_m \times V_m',U_m \times U_m',f_m \times f_m')$ constitutes an admissible gadget with a nice $G \times G'$-action. Indeed, to check the last condition of \cite[Definition 4.2.4]{MV-A1}, let $T \to X$ be a $G\times G'$-torsor in $\Sm_S$. Then the projection $(T \times U \times U')/(G\times G') \to T/(G\times G')$ factors as
\[
(T_1 \times U')/G' \to T_1/G'
\]
where $T_1 = (T \times U)/G$ followed by
\[
(T_2 \times U)/G \to T_2/G
\]
where $T_2 = T/G'$. Each morphism is an epimorphism in the Nisnevich topology by assumption, hence so is their composite.

Under this choice of a model for $\BB G$, we have
\begin{equation}
\label{eq:isomo_B_prod}
B_m(G \times G') = B_mG \times B_mG'.
\end{equation}
\end{para}

\begin{lemma}
\label{lemm:B_prod}
If $G,G'$ are linear algebraic groups over $S$, we have an isomorphism
\[
\BB(G \times G') \simeq \BB G \times \BB G' \in \Ho(S).
\]
\end{lemma}
\begin{proof}
Since the product with a given motivic space commutes with homotopy colimits, we have isomorphisms in $\Ho(S)$
\begin{align*}
\BB (G \times G')&\simeq \hocolim_m (B_mG \times B_mG') && \text{by \eqref{eq:isomo_B_prod} and \rref{p:colim_hocolim}}\\
&\simeq \hocolim_m \hocolim_d (B_mG \times B_dG') && \text{by a cofinality argument}\\
&\simeq \hocolim_m (B_mG \times \hocolim_d B_dG') \\
&\simeq (\hocolim_m B_mG) \times (\hocolim_d B_dG')\\
&\simeq \BB G \times \BB G'&&\text{by \rref{p:colim_hocolim}.}\qedhere
\end{align*}
\end{proof}

\subsection{Characters}
In this section, we discuss general facts relating the classifying space of a linear algebraic group $G$ to that of the kernel $H$ of a character of $G$, which will be applied to explicit situations in \S\ref{sect:compute_BG}.

\begin{para}
\label{p:C_m}
Let $G$ be a linear algebraic group over $S$, and fix a model for $\BB G$ (see \rref{p:def_BG}). Assume given a character of $G$, that is a morphism of algebraic groups $\chi \colon G \to \Gm$. Considering the right $G$-action on $\Ab^1$ given by letting $g\in G$ act via $\lambda \mapsto \chi(g)^{-1} \lambda$, we define for each $m \in \Nn \smallsetminus \{0\}$ a line bundle over $B_mG$:
\begin{equation}
\label{def:C_m}
C_m(\chi)= (E_mG \times \Ab^1)/G.
\end{equation}
The assignment $\chi \mapsto C_m(\chi)$ satisfies
\begin{equation}
\label{eq:prod_C_m}
C_m(\chi \chi') = C_m(\chi) \otimes C_m(\chi'),
\end{equation}
yielding group morphisms
\[
\Hom_{\text{alg.\ groups}}(G,\Gm) \to \Pic(B_mG) \quad ; \quad \chi \mapsto C_m(\chi).
\]

Let now $H \subset G$ be a closed subgroup, and consider the morphisms $B_m H \to B_mG$ defined in \rref{p:B_closed}. If $\chi|_H$ denotes the restriction of the character $\chi$ to $H$, then
\begin{equation}
\label{eq:pb_C_m_closed}
C_m(\chi) \times_{B_mG} B_mH \simeq C_m(\chi|_H).
\end{equation}

If $G'$ is a linear algebraic group over $S$, letting $\tilde{\chi}\colon G \times G' \to G \xrightarrow{\chi} \Gm$ be the induced character of $G \times G'$, we have
\begin{equation}
\label{eq:pb_C_m_prod}
C_m(\tilde{\chi}) = C_m(\chi) \times B_mG'.
\end{equation}
\end{para}

\begin{para}
\label{p:GH}
Let $G$ be a linear algebraic group over $S$, and $\chi$ a surjective character of $G$. Letting $H=\ker \chi$, we thus have an exact sequence of algebraic groups over $S$
\begin{equation}
1 \to H \to G \xrightarrow{\chi} \Gm \to 1.
\end{equation}
Let us fix a model for $\BB G$. As explained in \rref{p:B_closed}, this yields a model for $\BB H$, and morphisms $p_m \colon B_mH \to B_mG$ for $m \in \Nn \smallsetminus \{0\}$. By \rref{p:C_m}, we also have a line bundle $C_m(\chi)$ over $B_mG$, such that
\[
C_m(\chi)^{\circ} = ((E_mG \times \Ab^1)/G)^\circ = (E_mG \times \Gm)/G = ((E_mG)/H \times \Gm)/\Gm= (E_mG)/H =B_mH.
\]
In view of \rref{eq:E_circ_dist}, this yields a cofiber sequence in $\Spcp(S)$, for each $m\in \Nn\smallsetminus \{0\}$
\begin{equation}
\label{eq:BH_BG_m}
(B_mH)_+ \xrightarrow{{p_m}_+} (B_mG)_+ \to \Th_{B_mG}(C_m(\chi)).
\end{equation}
More generally (as in \rref{p:purity}), if $V \to B_mG$ is a vector bundle, we have a cofiber sequence in $\Spcp(S)$,
\begin{equation}
\label{eq:BH_BG_Th}
\Th_{B_mH}(V) \to \Th_{B_mG}(V) \to \Th_{B_mG}(C_m(\chi) \oplus V).
\end{equation}
\end{para}
\begin{para}
In the situation of \rref{p:GH}, let us define
\[
\Th_{\BB G}(C(\chi)) =\colim_m \Th_{B_mG}(C_m(\chi)) \in \Spcp(S).
\]
As in \rref{p:colim_hocolim}, this coincides with the homotopy colimit (the transition morphisms are again monomorphisms, being directed colimits of such). We will also write $\Th_{\BB G}(C(\chi)) \in \Spt(S)$ instead of $\Su \Th_{\BB G}(C(\chi))$. Taking the (homotopy) colimit of \eqref{eq:BH_BG_m} yields a cofiber sequence in $\Spcp(S)$
\begin{equation}
\label{eq:BH_BG}
(\BB H)_+ \to (\BB G)_+ \to \Th_{\BB G}(C(\chi)).
\end{equation}
\end{para}

\begin{para}
In the situation of \rref{p:GH}, assume that the model for $\BB G$ is pointed. Then we have a commutative diagram of $S$-schemes with cartesian squares
\begin{equation}
\label{squ:e_i_j}
\begin{gathered}
\xymatrix{
G\ar[r]^{\chi} \ar[d]^{e_1} & \Gm \ar[d]^{j_1} \ar[r] & S \ar[d]^{i_1}\\ 
E_1G \ar[r] & B_1H \ar[r] & B_1G
}
\end{gathered}
\end{equation}
where $e_1$ is induced by the $S$-point and the $G$-action on $E_1G$, and $j_1$, resp.\ $i_1$, is obtained by taking the $H$-quotient, resp.\ $G$-quotient of $e_1$. Composing $i_1$ and $j_1$ with the natural maps $B_1 G \to \BB G$ and $B_1 H \to\BB H$ respectively, we obtain maps $i \colon S \to \BB G$ and $j \colon \Gm \to \BB H$ in $\Spc(S)$. Note that $i$ is the map described in \rref{p:pointed}, and that the left-hand cartesian square in \eqref{squ:e_i_j} shows that the map $j \colon \Gm \to \BB H$ classifies the $H$-torsor $\chi \colon G \to \Gm$. 

The right-hand cartesian square in \eqref{squ:e_i_j} shows that the $\Gm$-torsor $B_1 H \to B_1G$ pulls back to the trivial torsor along $i_1$, which yields a trivialisation of the line bundle $i_1^*C_1(\chi)$ over $S$, and thus a morphism in $\Spcp(S)$
\[
t \colon T=\Th_S(1) \to \Th_{B_1G}(C_1(\chi)) \to \Th_{\BB G}(C(\chi)).
\]
We thus obtain a commutative diagram in $\Spcp(S)$, whose rows are cofiber sequences
\begin{equation}
\label{diag:BH_BG}
\begin{split}
\xymatrix{
(\BB H)_+\ar[r]  & (\BB G)_+ \ar[r] & \Th_{\BB G}(C(\chi)) \\ 
(\Gm)_+ \ar[r] \ar[u]^{j_+}& S_+ \ar[u]^{i_+} \ar[r] & T\ar[u]^{t} 
}
\end{split}
\end{equation}
\end{para}

\begin{para}
In the situation of \rref{p:GH}, using \rref{prop:L_circ} for the line bundle $C_m(\chi) \to B_mG$, and applying the functor $\Spcp(B_mG)[\eta^{-1}] \to \Spcp(S)[\eta^{-1}]$ of \rref{p:sharp}, we have homotopy cocartesian squares in $\Spcp(S)[\eta^{-1}]$
\[ \xymatrix{
\T \wedge (B_m H)_+\ar[d] \ar[r] & \Th_{B_m G}(C(\chi_m)) \ar[d] \\ 
\T \wedge (B_mG)_+ \ar[r] & {*}
}\]
which are compatible with the transition maps as $m$ varies by \rref{p:pi_pb}. Taking the homotopy colimit, and proceeding as in the proof of \rref{cor:L_circ}, we obtain an isomorphism in $\SH(S)[\eta^{-1}]$
\begin{equation}
\label{eq:splitting_BG_BH}
\Sup \BB H \simeq \Sup \BB G \vo \Sigma^{-2,-1}\Th_{\BB G}(C(\chi)) \in \SH(S)[\eta^{-1}].
\end{equation}
\end{para}

\section{Computations of classifying spaces}
\label{sect:compute_BG}

\subsection{Diagonalisable groups}

Using the embeddings $\Pp^k \subset \Pp^{k+1}$ of \rref{p:Pk} for $k \in \Nn$, we define, for $n \in \Zz$
\[
\Pp^\infty = \hocolim_k \Pp^k \in \Spc(S) \quad \text{and} \quad \Th_{\Pp^{\infty}}(\Oc(n)) = \hocolim_k \Th_{\Pp^k}(\Oc(n)) \in \Spcp(S),
\]
and as usual write $\Th_{\Pp^\infty}(\Oc(n)) \in \Spt(S)$ instead of $\Su  \Th_{\Pp^\infty}(\Oc(n))$. We have a natural map $\iota_\infty \colon S=\Pp^0 \to \Pp^\infty$ in $\Spc(S)$. For each $n \in \Zz$, the line bundle $\Oc(n)$ over $\Pp^0$ admits a canonical trivialisation, so that \eqref{eq:Sigma_triv_P} yields a canonical map in $\Spt(S)$
\begin{equation}
\label{eq:Sigma_n}
\Sigma^{2,1} \Un_S \to \Th_{\Pp^\infty}(\Oc(n)).
\end{equation}

\begin{proposition}
\label{prop:P_infty}
Let $n\in \Zz$. The following hold in $\SH(S)[\eta^{-1}]$:
\begin{enumerate}[(i)]
\item \label{prop:P_infty:1}
The morphism $\iota_\infty$ induces an isomorphism $\Un_S \simeq \Sup\Pp^\infty$.

\item \label{prop:P_infty:2}
If $n$ is odd, then $\Th_{\Pp^\infty}(\Oc(n)) = 0$.

\item \label{prop:P_infty:3}
If $n$ is even, then \eqref{eq:Sigma_n} induces an isomorphism $\Sigma^{2,1}\Un_S \simeq \Th_{\Pp^\infty}(\Oc(n))$.
\end{enumerate}
\end{proposition}
\begin{proof}
We apply \rref{prop:Th_Pk} with $\Dc = (n)$, and so $\Oc(\Dc)=\Oc(n)$. We obtain that $\Sigma^{2,1}\Un_S \to \Th_{\Pp^k}(\Oc(n))$ is a weak equivalence in $\Spt(S)[\eta^{-1}]$ when $n,k$ are even. Taking the homotopy colimit over $k$ yields a weak equivalence $\Sigma^{2,1}\Un_S \to \Th_{\Pp^\infty}(\Oc(n))$ in $\Spt(S)[\eta^{-1}]$ when $n$ is even. This proves \eqref{prop:P_infty:3}. The other statements are deduced in a similar way from \rref{prop:Th_Pk}.
\end{proof}

\begin{para}
\label{p:BGm}
Consider the model for $\BB \Gm$ described in \rref{p:Totaro_model} with $p=n=1$, under the identification $\Gm = \GL_1$. Then $B_m\Gm = \Pp^{m-1}$, and thus $\BB \Gm = \Pp^{\infty}$. Furthermore, the line bundle $C_m(\id_{\Gm})$ over $B_mG$ defined in \eqref{def:C_m} may be identified with the tautological bundle $\Oc(-1)$ over $\Pp^{m-1}$.
\end{para}

\begin{theorem}
\label{th:B_mu}
Let $n\in \Nn \smallsetminus \{0\}$. The following hold in $\SH(S)[\eta^{-1}]$:
\begin{enumerate}[(i)]
\item The natural morphism $\Un_S \to \Sup\BB\Gm$ is an isomorphism.

\item If $n$ is odd, the natural morphism $\Un_S \to \Sup\BB\mu_n$ is an isomorphism.

\item
\label{th:B_mu:even}
If $n$ is even, the morphism $\Gm \to \BB \mun$ classifying the $\mu_n$-torsor $\Gm \to \Gm$ given by taking $n$-th powers induces an isomorphism $\Sup\Gm \simeq \Sup\BB\mu_n$.
\end{enumerate}
\end{theorem}
\begin{proof}
Let us consider the model for $\BB\Gm$ described in \rref{p:BGm}, where $B_m\Gm = \Pp^{m-1}$ and $\BB \Gm = \Pp^\infty$. Then the first statement follows from \dref{prop:P_infty}{prop:P_infty:1}.

Next consider the character $n \colon \Gm \to \Gm$ given by taking $n$-th powers. Its kernel is $\mu_n$, and the line bundle $C_m(n)$ over $B_m\Gm$ (defined in \eqref{def:C_m}) corresponds to the line bundle $\Oc(-n)$ over $\Pp^{m-1}$ (this may be seen for instance by combining \rref{p:BGm} with \eqref{eq:prod_C_m}). So we are in the situation of \rref{p:GH} with $G=\Gm,\chi=n,H=\mu_n$. Thus \eqref{eq:BH_BG} yields a distinguished triangle in $\SH(S)$
\[
\Sup \BB \mu_n \to \Sup \BB \Gm \to \Th_{\Pp^\infty}(\Oc(-n)) \to \Sigma^{1,0} \Sup \BB \mu_n.
\]

If $n$ is odd, then $\Th_{\Pp^\infty}(\Oc(-n))=0$ in $\SH(S)[\eta^{-1}]$ by \dref{prop:P_infty}{prop:P_infty:2}, and the above distinguished triangle shows that the morphism $\Sup \BB \mu_n \to \Sup \BB \Gm$ is an isomorphism in  $\SH(S)[\eta^{-1}]$. Thus the second statement follows from the first.

Assume that $n$ is even. Then in the diagram \eqref{diag:BH_BG} the maps $i_+$ and $t$ become isomorphisms in $\SH(S)[\eta^{-1}]$ by \dref{prop:P_infty}{prop:P_infty:1} and \dref{prop:P_infty}{prop:P_infty:3}, hence $\Sup j \colon \Sup \Gm \to \Sup \BB \mu_n$ is also one. As observed in \rref{p:GH}, the latter is induced by the $\mu_n$-torsor $n\colon \Gm \to \Gm$.
\end{proof}

\begin{remark}
Combining \rref{th:B_mu} with \rref{lemm:B_prod}, we have thus obtained a ``computation'' in $\SH(S)[\eta^{-1}]$ of the classifying space $\BB G$ of every finitely generated diagonalisable group $G$.
\end{remark}

\subsection{\texorpdfstring{$\SL$}{SL} versus \texorpdfstring{$\SL^c$}{SLc}}
Let $n \in \Nn \smallsetminus \{0\}$. Consider the character
\[
\nu_n \colon \GL_n \times \Gm \to \Gm \quad ; \quad (M,t) \mapsto t^{-2} \det M.
\]
By definition (see \cite[\S3]{PW-BO}), we have $\SL_n^c = \ker \nu_n$. We view $\SL_n$ as a subgroup of $\SL^c_n$ via the mapping $M \mapsto (M,1)$.

\begin{proposition}
\label{prop:SLc_SL}
For $n \in \Nn\smallsetminus \{0\}$ the inclusion $\SL_n \subset \SL_n^c$ induces an isomorphism 
\[
\Sup \BSL_n \xrightarrow{\sim} \Sup \BSL^c_n  \text{ in $\SH(S)[\eta^{-1}]$.}
\]
\end{proposition}
\begin{proof}
The character
\[
\delta_n \colon \SL_n^c \to \Gm \quad ; \quad (M,t) \mapsto t
\]
is surjective (recall that $n \geq 1$), and satisfies $\SL_n = \ker \delta_n$.

Set $P_n=\GL_n \times \Gm$, and denote by $q_n\colon P_n \to \Gm$ the second projection. Let us fix an arbitrary model for $\BGL_n$, but choose the model for $\BB\Gm$ described in \rref{p:BGm}, so that $B_m \Gm=\Pp^{m-1}$. Recall from \rref{eq:isomo_B_prod} that this yields a model for $\BB P_n$, such that 
\[
B_mP_n = (B_m\GL_n) \times (B_m\Gm) = (B_m\GL_n) \times \Pp^{m-1}.
\]
Letting $g_n \colon P_n \to \GL_n$ be the first projection, we have, as characters $P_n \to \Gm$
\[
\nu_n = g_n^*({\det}_n) \cdot q_n^*(\id_{\Gm})^{-2},
\]
where $\det_n \colon \GL_n \to \Gm$ denotes the determinant morphism. It follows from \eqref{eq:prod_C_m} and \eqref{eq:pb_C_m_prod} that, as line bundles over $B_m P_n$, we have in the notation of \rref{def:C_m}
\[
C_m(\nu_n) \simeq C_m({\det}_{n}) \boxtimes C_m(\id_{\Gm})^{\otimes -2} \quad ; \quad C_m(q_n) \simeq 1 \boxtimes C_m(\id_{\Gm}).
\]
Recall from \rref{p:BGm} that the line bundle $C_m(\id_{\Gm}) \to B_m\Gm$ corresponds to $\Oc(-1) \to \Pp^{m-1}$. Applying \rref{lemm:Th_odd} to the projective bundle $B_mP_n \to B_m \GL_n$, for $m$ even we have 
\begin{equation}
\label{eq:Th_C_m}
\Th_{B_mP_n}(C_m(\nu_n) \oplus C_m(q_n))=0= \Th_{B_mP_n}(C_m(q_n)) \in \SH(S)[\eta^{-1}].
\end{equation}

Since the character $\delta_n$ is the restriction of $q_n \colon P_n \to \Gm$, it follows from \eqref{eq:pb_C_m_closed} that the line bundle $C_m(\delta_n)$ over $B_m\SL_n^c$ is the pullback of $C_m(q_n)$ over $B_mP_n$. By \rref{eq:BH_BG_Th}, we have a distinguished triangle in $\SH(S)$
\begin{align*}
\Th_{B_m\SL_n^c}(C_m(\delta_n)) \to \Th_{B_mP_n}(C_m(q_n)) \to& \Th_{B_mP_n}(C_m(\nu_n) \oplus C_m(q_n)) \\
&\quad \quad \to   \Sigma^{1,0}\Th_{B_m\SL_n^c}(C_m(\delta_n)).
\end{align*}
so that, in view of \eqref{eq:Th_C_m} 
\[
\Th_{B_m\SL_n^c}(C_m(\delta_n)) =0 \in \SH(S)[\eta^{-1}] \quad \text{for $m$ even}.
\]
Now the distinguished triangle in $\SH(S)$ (see \rref{eq:BH_BG_m})
\[
\Sup B_m\SL_n \to \Sup B_m\SL_n^c \to \Th_{B_m\SL_n^c}(C_m(\delta_n)) \to \Sigma^{1,0} \Sup B_m\SL_n
\]
implies that, for $m$ even, the natural map induces an isomorphism
\[
\Sup B_m\SL_n \xrightarrow{\sim} \Sup B_m\SL_n^c \quad \text{in $\SH(S)[\eta^{-1}]$}.
\]
The statement follows by taking the homotopy colimit.
\end{proof}

\begin{remark}
\label{p:Ana_BSL}
Let $A \in \SH(S)$ be an $\eta$-periodic commutative ring spectrum, and consider the corresponding cohomology theory $\Aa(-)$ (see \rref{p:coh}). Then by \rref{prop:SLc_SL}, we have a natural isomorphism
\[
\Aa(\BSL_n^c) \simeq \Aa(\BSL_n).
\]
If $A$ is $\SL$-oriented  (see \rref{def:SL-orientation}) and $S=\Spec k$ with $k$ a field of characteristic not two, Ananyevskiy computed in \cite[Theorem~10]{Ananyevskiy-SL_PB} that
\[
\Aa(\BSL_n) = 
\begin{cases}
\Aa(S)[[p_1,\ldots,p_{r-1},e]]_h & \text{if $n=2r$ with $r\in\Nn \smallsetminus \{0\}$}\\
\Aa(S)[[p_1,\ldots,p_r]]_h & \text{if $n=2r+1$ with $r\in \Nn$},
\end{cases}
\]
where $p_i$ has degree $(4i,2i)$ and $e$ has degree $(2r,r)$ (here the notation $R[[x_1,\ldots,x_m]]_h$ refers to the homogeneous power series ring in $m$ variables over the graded ring $R$, see \cite[Definition~27]{Ananyevskiy-SL_PB}). This computation remains valid (with exactly the same arguments) when $S$ is an arbitrary noetherian scheme of finite dimension, under the assumption that $2$ is invertible in $S$. Removing that last assumption seems to require a modification of the arguments of \cite{Ananyevskiy-SL_PB}, the problem being with \cite[Lemma~6]{Ananyevskiy-SL_PB} (which is used to prove \cite[Theorem~9]{Ananyevskiy-SL_PB}).
\end{remark}

\subsection{\texorpdfstring{$\GL$}{GL} and \texorpdfstring{$\SL$}{SL}}
In this section, we compare the classifying spaces $\BGL_{2r}$, $\BGL_{2r+1}$, $\BSL_{2r}$, $\BSL_{2r+1}$.

\begin{para}
\label{p:BGL_n_1}
Recall that under the model for $\BGL_n$ described in \rref{p:Totaro_model} for $p=n$, the scheme $B_m \GL_n$ is identified with the grassmannian $\Gr(n,nm)$. The closed immersion $\Gr(n,nm) \to \Gr(n+1,(n+1)m)$ mapping a subbundle $E \subset (1^{\oplus m})^{\oplus n}$ to $E \oplus 1 \subset (1^{\oplus m})^{\oplus n} \oplus 1^{\oplus m} = (1^{\oplus m})^{\oplus n+1}$, where the inclusion $1 \subset 1^{\oplus m}$ is given by the vanishing of the last $m-1$ coordinates, induces a morphism $f_m \colon B_m \GL_n \to B_m\GL_{n+1}$ which is compatible with the transition maps as $m$ varies. This yields a morphism in $\Spc(S)$
\begin{equation}
\label{eq:BGL_n_1}
\BGL_n \to \BGL_{n+1}.
\end{equation}
\end{para}

\begin{para}
\label{p:p_q}
For integers $u,v,w \in \Nn$, we denote by $\Gr(u\subset v,w)$ the flag variety of subbundles $P \subset Q \subset 1^{\oplus w}$ with $\rank P=u$ and $\rank Q=v$. Let $r,s \in \Nn$, and consider the morphisms
\[
\Gr(2r,s) \xleftarrow{p} \Gr(2r\subset 2r+1,s) \xrightarrow{q} \Gr(2r+1,s)
\]
given by mapping a flag $P \subset Q$ to $P$, resp.\ $Q$.

For $n \in \{2r,2r+1\}$, let us denote by $\Uni_n \subset 1^{\oplus s}$ the tautological rank $n$ subbundle over $\Gr(n,s)$, and write $\Quo_n = 1^{\oplus s}/\Uni_n$. Then the morphism $p$ may be identified with the projective bundle $\Pp(\Quo_{2r})$, and the morphism $q$ is the projective bundle $\Pp(\Uni_{2r+1}^\vee)$.
\end{para}

\begin{proposition}
\label{prop:BGL_ev_odd}
The map $\BGL_{2r} \to \BGL_{2r+1}$ of \eqref{eq:BGL_n_1} becomes an isomorphism in $\SH(S)[\eta^{-1}]$.
\end{proposition}
\begin{proof}
Let $n=2r$. For $m\in \Nn \smallsetminus \{0\}$, consider the commutative diagram in $\Sm_S$
\[ \xymatrix{
&\Gr(n,nm)\ar[ld]_{g_m}\ar[rd]^{f_m} \ar[d]^{j_m} &  \\ 
\Gr(n,(n+1)m) &\Gr(n \subset n+1,(n+1)m)\ar[l]^-{p_m}\ar[r]_-{q_m} & \Gr(n+1,(n+1)m)
}\]
where the morphism $j_m$ is given by mapping $E \subset (1^{\oplus m})^{\oplus n}$ to
\[
E \subset E \oplus 1 \subset (1^{\oplus m})^{\oplus n} \oplus 1^{\oplus m} = (1^{\oplus m})^{\oplus n+1},
\]
with the inclusion $1 \subset 1^{\oplus m}$ given by the vanishing of the $m-1$ last coordinates. Here the morphisms $p_m,q_m$ are the morphisms $p,q$ described in \rref{p:p_q} when $s=(n+1)m$. The morphism $f_m$ is the one described in \rref{p:BGL_n_1}, and the morphism $g_m$ is induced by the inclusion
\begin{equation}
\label{eq:inc_m_n}
(1^{\oplus m})^{\oplus n} =(1^{\oplus m})^{\oplus n} \oplus 0 \subset (1^{\oplus m})^{\oplus n} \oplus 1^{\oplus m} = (1^{\oplus m})^{\oplus n+1}.
\end{equation}
The morphisms of this diagram are compatible with the transition maps as $m$ varies, induced by the inclusions $1^{\oplus m} \subset 1^{\oplus m+1}$ given by the vanishing of the last coordinate.

The morphism $q_m$ is a $\Pp^n$-bundle, hence is an isomorphism in $\SH(S)[\eta^{-1}]$ by \rref{cor:PE_odd} (recall that $n=2r$ is even). The morphism $p_m$ is a $\Pp^{(n+1)m-n-1}$-bundle, hence is also an isomorphism in $\SH(S)[\eta^{-1}]$ when $m$ is odd by \rref{cor:PE_odd}.

In the notation of \rref{p:Totaro_model} the morphism $g_m$ is the $\GL_n$-quotient of the morphism $U_{m,n} \to U_{m,n+1}$ induced by \eqref{eq:inc_m_n}. Therefore it follows from \rref{p:def_BG} that the map $\colim_m g_m$ is a weak equivalence of motivic spaces. 

Applying the functor $\Sup \colon \Sm_S \to \Spt(S)[\eta^{-1}]$ to the above diagram, and taking the homotopy colimit over $m$, we thus obtain a commutative diagram in $\Spt(S)[\eta^{-1}]$, where all maps are weak equivalences. Since the map \eqref{eq:BGL_n_1} is obtained as $\colim_m f_m$, the proposition follows. 
\end{proof}

\begin{para}
Let $n \in \Nn \smallsetminus \{0\}$. The group $\SL_n$ is the kernel of the determinant morphism $\det_n \colon \GL_n \to \Gm$, which is surjective (as $n\geq 1$). We are thus in the situation of \rref{p:GH}, so that we have by \eqref{eq:splitting_BG_BH}, a splitting
\begin{equation}
\label{eq:BSL_BSL_split}
\Sup \BSL_n = \Sup \BGL_n \oplus \Sigma^{-2,-1}\Th_{\BGL_n}(C({\det}_n)) \in \SH(S)[\eta^{-1}].
\end{equation}
\end{para}

\begin{para}
Using the model for $\BGL_n$ described in \rref{p:Totaro_model} with $p=n$, the variety $B_m \GL_n$ coincides with the grassmannian $\Gr(n,nm)$. Observe that the tautological bundle $\Uni_n$ over this variety is isomorphic to the quotient $(E_m\GL_n \times \Ab^n)/\GL_n$, where the right $\GL_n$-action on $\Ab^n$ is given by letting $\varphi \in \GL_n$ act via $v \mapsto \varphi^{-1}(v)$. The $\GL_n$-equivariant isomorphism $\det(E_m\GL_n \times \Ab^n) \simeq E_m\GL_n \times \Ab^1$, where the right $\GL_n$-action on $\Ab^1$ is given by letting $\varphi \in \GL_n$ act via $\lambda \mapsto \det_n(\varphi^{-1}) \lambda$, yields an isomorphism of line bundles over $B_m\GL_n$
\begin{equation}
\label{eq:det_Uni}
\det \Uni_n \simeq C_m({\det}_n).
\end{equation}
\end{para}

\begin{proposition}
\label{prop:BSL_BGL_odd}
Let $r \in \Nn$. Then in $\SH(S)[\eta^{-1}]$, the natural morphism $\BSL_{2r} \to \BGL_{2r}$ acquires a section, and the natural morphism $\BSL_{2r+1} \to \BGL_{2r+1}$ becomes an isomorphism.
\end{proposition}
\begin{proof}
The first statement follows from \rref{eq:BSL_BSL_split}. Let us prove the second. We use the model for $\BGL_{2r+1}$ described in \rref{p:Totaro_model} with $n=p=2r+1$, so that $B_m \GL_{2r+1} = \Gr(2r+1,s)$ where $s=(2r+1)m$. We consider the situation of \rref{p:p_q}, and use the notation thereof. We have an exact sequence of vector bundles over $Y_m=\Gr(2r \subset 2r+1,s)=\Pp(\Quo_{2r})$
\[
0 \to p^*\Uni_{2r} \to q^*\Uni_{2r+1} \to \Oc_{\Pp(\Quo_{2r})}(-1) \to 0,
\]
Taking determinants and using \eqref{eq:det_Uni} we obtain an isomorphism of line bundles
\begin{equation}
\label{eq:det_tensor}
p^*(\det \Uni_{2r}) \otimes \Oc_{\Pp(\Quo_{2r})}(-1) \simeq q^*(\det \Uni_{2r+1}) \simeq q^*C_m({\det}_{2r+1}).
\end{equation}
When $m$ is even, the vector bundle $\Quo_{2r}=1^{\oplus s}/\Uni_{2r}$ has even rank, hence it follows from \rref{lemm:Th_odd} and \eqref{eq:det_tensor} that $\Th_{Y_m}(q^*C_m(\det_{2r+1}))=0$ in $\SH(S)[\eta^{-1}]$. Since $q \colon Y_m \to \Gr(2r+1,s)$ is a $\Pp^{2r}$-bundle, it then follows from \rref{prop:PE_odd} (applied with $V=C_m(\det_{2r+1})$) that $\Th_{\Gr(2r+1,s)}(C_m(\det_{2r+1}))=0$ in $\SH(S)[\eta^{-1}]$ for $m$ even. Taking the homotopy colimit over $m$, we deduce that $\Th_{\BGL_{2r+1}}(C(\det_{2r+1}))=0$ in $\SH(S)[\eta^{-1}]$, and the second statement of the proposition follows from \eqref{eq:BSL_BSL_split}.
\end{proof}

\begin{remark}
Let $A \in \SH(S)$ be an $\eta$-periodic $\SL$-oriented commutative ring spectrum (see \rref{def:SL-orientation}), and consider the corresponding cohomology theory $\Aa(-)$ (see \rref{p:coh}). Assume that $2$ is invertible in $S$. Combining \rref{prop:BSL_BGL_odd} and \rref{prop:BGL_ev_odd} with Ananyevskiy's computation of $\Aa(\BSL_{2r+1})$ (see \rref{p:Ana_BSL}), we recover Levine's computation \cite[Theorem~4.1]{Levine-motivic_Euler}
\[
\Aa(\BGL_{2r}) = \Aa(\BGL_{2r+1}) = \Aa(S)[[p_1,\ldots,p_r]]_h.
\]
(Note that this permits to remove some of the technical assumptions present in the statement of \cite[Theorem~4.1]{Levine-motivic_Euler}.)
\end{remark}

\appendix
\section{An invariant of \texorpdfstring{$\mu_2$}{\textmu 2}-torsors}
\numberwithin{theorem}{section}
\numberwithin{lemma}{section}
\numberwithin{proposition}{section}
\numberwithin{corollary}{section}
\numberwithin{example}{section}
\numberwithin{definition}{section}
\numberwithin{remark}{section}

We have proved in \rref{th:B_mu} that in $\SH(S)[\eta^{-1}]$, for $n\in \Nn$
\[
\Sup \BB \mu_n =
\begin{cases}
\Un_S & \text{if $n=0$ (i.e. $\mu_n = \Gm$) or $n$ is odd,}\\
\Sup \Gm = \Un_S \vo \Un_S & \text{if $n>0$ is even.}
\end{cases}
\]
Thus, when $n>0$ is even, there is essentially one nontrivial invariant of $\mu_n$-torsors over $S$ in $\SH(S)[\eta^{-1}]$, in the form of an element of $\End_{\SH(S)[\eta^{-1}]}(\Un_S)$. Moreover, it also follows from \rref{th:B_mu} that the morphism $\BB \mu_2 \to \BB \mu_{2r}$ becomes an isomorphism in $\SH(S)[\eta^{-1}]$ for $r>0$. Therefore the above-mentioned invariant of $\mu_n$-torsors is induced by an invariant of $\mu_2$-torsors, which is however not really explicit from this description. In this section we provide an explicit construction of this invariant (the connection with the above discussion is made in \rref{rem:coh} below).  

\begin{para}
\label{p:Thom_dual}
Let $L\to S$ be a line bundle. As observed by Ananyevskiy \cite[Lemma~4.1]{Ana-SL}, the isomorphism of $S$-schemes $L^\circ \xrightarrow{\sim} (L^\vee)^\circ$, given locally by $l \mapsto l^\vee$ where $l^\vee(l)=1$, induces an isomorphism in $\Hop(S)$
\[
\sigma_L \colon \Th_S(L) \xrightarrow{\sim} \Th_S(L^\vee).
\]
\end{para}

\begin{para}
If $\varphi \colon L \to M$ is an isomorphism of line bundles over $X$, we have (see \rref{p:Thom_funct})
\begin{equation}
\label{eq:T_sigma}
\sigma_L = \Th(\varphi^\vee) \circ \sigma_M \circ \Th(\varphi).
\end{equation}
\end{para}

\begin{definition}
\label{p:def_mu_2_torsor}
It will be convenient to think of a $\mu_2$-torsor over $S$ as a pair $(L,\lambda)$, where $L\to S$ is a line bundle, and $\lambda \colon L \xrightarrow{\sim} L^\vee$ is an isomorphism of line bundles over $S$. Isomorphisms $(L,\lambda) \to (L',\lambda')$ are given by isomorphisms of line bundles $\varphi \colon L \xrightarrow{\sim} L'$ such that $\lambda = \varphi^\vee \circ \lambda' \circ \varphi$. The set of isomorphism classes of $\mu_2$-torsors is denoted $H^1_{et}(S,\mu_2)$; it is endowed with a group structure induced by the tensor product of line bundles.
\end{definition}

\begin{definition}
\label{def:alpha}
Consider a $\mu_2$-torsor, given by a line bundle $L\to S$ and an isomorphism $\lambda \colon L\xrightarrow{\sim} L^\vee$. Let us consider the composite isomorphism in $\SH(S)$ (see \rref{p:Thom_dual} for the definition of $\sigma_L$, and \rref{p:Thom_funct} for that of $\Th(\lambda)$)
\[
\Th_S(L) \xrightarrow{\Th(\lambda)} \Th_S(L^\vee) \xrightarrow{\sigma_L^{-1}} \Th_S(L).
\]
This yields an element $a_{(L,\lambda)} = \Sigma^{-L}(\sigma_L^{-1} \circ \Th(\lambda)) \in \Aut_{\SH(S)}(\mathbf{1}_S)$. We define
\[
\alpha(L,\lambda) = (a_{(1,\can)})^{-1} \circ a_{(L,\lambda)} \in \Aut_{\SH(S)}(\Un_S),
\]
where $\can \colon 1 \to 1^\vee$ is the canonical isomorphism of line bundles over $S$. (The element $a_{(1,\can)}$ corresponds to the element $\epsilon$ of \cite[\S6.1]{Morel-Intro_A1}.)
\end{definition}

\begin{para}
\label{p:funct_alpha}
This construction is compatible with pullbacks, in the sense that if $f\colon R \to S$ is a morphism of noetherian schemes of finite dimension, and $(L,\lambda)$ a $\mu_2$-torsor, then the composite in $\SH(R)$
\[
\Un_R = f^*\Un_S \xrightarrow{f^*\alpha(L,\lambda)} f^*\Un_S =\Un_R
\]
is $\alpha(f^*L,f^*\lambda)$.
\end{para}

\begin{example}
\label{ex:L_trivial}
Consider a $\mu_2$-torsor $(L,\lambda)$ where $L=1$ is the trivial line bundle. Then $\lambda = u \can$ for some $u \in H^0(S,\Gm)$, hence $\Th(\lambda) = \Th(\can) \circ \Th(u\id_1)$. Therefore (see \rref{p:langle_rangle})
\begin{align*}
\alpha(1,\lambda) 
&=  \Sigma^{-2,-1}((\sigma_1^{-1} \circ \Th(\can))^{-1} \circ (\sigma_1^{-1} \circ \Th(\lambda))  \\ 
&= \Sigma^{-2,-1}(\Th(\can)^{-1} \circ \Th(\lambda))\\
&= \Sigma^{-2,-1} \Th(u\id_1)\\
&= \langle u\rangle.
\end{align*}
\end{example}

\begin{proposition}
\label{prop:mu_2_isom}
The assignment $(L,\lambda) \mapsto \alpha(L,\lambda)$ induces a morphism of pointed sets
\[
\alpha \colon H^1_{et}(S,\mu_2) \to \Aut_{\SH(S)}(\Un_S).
\]
\end{proposition}
\begin{proof}
By construction we have $\alpha(1,\can)=\id$. Consider an isomorphism of $\mu_2$-torsors $(L,\lambda) \xrightarrow{\sim} (M,\mu)$ (see \rref{p:def_mu_2_torsor}), given by an isomorphism $\varphi \colon L \xrightarrow{\sim} M$. Let us set $a_L = a_{(L,\lambda)}$ and $a_M = a_{(M,\mu)}$ (see \rref{def:alpha}). Then, in $\Aut_{\SH(S)}(\Th_S(L))$,
\begin{align*}
\Sigma^La_L = \sigma_L^{-1} \circ \Th(\lambda)
&= \sigma_L^{-1} \circ \Th(\varphi^\vee) \circ \Th(\mu) \circ \Th(\varphi) && \text{(as $\lambda = \varphi^\vee \circ \mu \circ \varphi$)} \\ 
&= \Th(\varphi)^{-1} \circ \sigma_M^{-1} \circ \Th(\mu) \circ \Th(\varphi) && \text{by \eqref{eq:T_sigma}}\\
&= \Th(\varphi)^{-1} \circ (\Sigma^Ma_M) \circ \Th(\varphi) \\
&=  \Th(\varphi)^{-1} \circ \Th(\varphi) \circ (\Sigma^L a_M) && \text{(by \eqref{eq:Sigma_T})}\\
&= \Sigma^La_M,
\end{align*}
whence $a_L = a_M$, and $\alpha(L,\lambda) = \alpha(M,\mu)$.
\end{proof}

\begin{proposition}
The assignment $(L,\lambda) \mapsto \alpha(L,\lambda)$ induces a group morphism
\[
\alpha \colon H^1_{et}(S,\mu_2) \to \Aut_{\SH(S)[\eta^{-1}]}(\Un_S).
\]
\end{proposition}
\begin{proof}
Consider $\mu_2$-torsors given by line bundles $L,M$ over $S$, and isomorphisms $L\xrightarrow{\sim} L^\vee$ and $M \xrightarrow{\sim} M^\vee$. As the functor $f^* \colon \SH(S)[\eta^{-1}] \to \SH(L^\circ)[\eta^{-1}]$ is faithful \rref{cor:L_circ_faithful}, by functoriality \rref{p:funct_alpha} we may assume that the line bundle $L$ is trivial. Similarly, we may assume that $M$ is also trivial. By \rref{prop:mu_2_isom}, we may assume that $L=M=1$. Then, in view of \rref{ex:L_trivial}, the statement follows from  \eqref{eq:<>_mult}.
\end{proof}

\begin{remark}
\label{rem:coh}
Let $\rho \colon \Gm \to \BB \mu_2$ be the map classifying the $\mu_2$-torsor $(1,t\id)$ over $\Gm$, where $t\in H^0(\Gm,\Gm)$ is the tautological section. Recall from \dref{th:B_mu}{th:B_mu:even} that $\Sup \rho \colon \Sup \Gm \to \Sup \BB \mu_2$ is an isomorphism in $\SH(S)[\eta^{-1}]$. If a $\mu_2$-torsor $(L,\lambda)$ is classified by the map $f \colon S \to \BB \mu_2$, we claim that $\alpha(L,\lambda)$ is the composite in $\SH(S)[\eta^{-1}]$ (the map $\pi_1$ was defined in \rref{def:pi})
\begin{equation}
\label{eq:composite}
\Un_S \xrightarrow{\Sup f} \Sup \BB \mu_2 \xrightarrow{(\Sup \rho)^{-1}} \Sup \Gm \xrightarrow{\Sigma^{-2,-1}\Sup \pi_1} \Un_S.
\end{equation}
Indeed, applying the functor $\SH(S)[\eta^{-1}] \to \SH(L^\circ)[\eta^{-1}]$ which is faithful by \rref{cor:L_circ_faithful}, and using \rref{prop:mu_2_isom}, we may assume that $L=1$. Then $\lambda$ is given by multiplication by an element $u \in H^0(S,\Gm)$, and the map $f$ factors as $S \xrightarrow{u} \Gm \xrightarrow{\rho} \BB \mu_2$. Denoting by $p \colon \Gm \to S$ is the projection, the composite \eqref{eq:composite} is given by
\[
\Un_S \xrightarrow{\Sup u} \Sup \Gm \xrightarrow{\langle t \rangle} \Sup \Gm \xrightarrow{\Sup p} \Un_S,
\]
which coincides with $\langle u \rangle \in \Aut_{\SH(S)[\eta^{-1}]}(S)$. Thus the claim follows from \rref{ex:L_trivial}.
\end{remark}

\end{document}